\documentclass[12pt]{amsart}
\usepackage{graphicx}
\usepackage{epstopdf}
\usepackage{commath}
\usepackage{xcolor,tikz}
\usepackage{hyperref}
\usepackage[hypcap]{caption}
\newcommand{\Cross}{\mathbin{\tikz [x=1.4ex,y=1.4ex,line width=.2ex] \draw (0,0) -- (1,1) (0,1) -- (1,0);}}%

\newcommand{\Wg}{\mathrm{Wg}}
\def\tr{\mathrm{tr}}
\def\Tr{\mathrm{Tr}}
\def\<{\langle}
\def\>{\rangle}
\def\P{{\mathbb P}}

\def\E{{\mathbb E}}

\newcommand{\ab}{\allowbreak}

\newcommand{\cP}{\mathcal{P}}

\newcommand{\VN}{\mathcal{V}_{\overrightarrow{\sigma}, \varepsilon, N }{(p,q)}}
\newcommand{\ANm}[1]{\mathcal{A}_{ \overrightarrow{\sigma}
\varepsilon, N}^{(p,q)}(#1)}
\newcommand{\AN}{\mathcal{A}^{(p,q)}_{ \overrightarrow{\sigma}, \varepsilon, N }}
\newcommand{\FNm}[1]{\mathcal{F}_{\overrightarrow{\sigma}, \varepsilon, N }^{(p,q)}(#1)}
\newcommand{\FN}{\mathcal{F}^{(p,q)}_N}
\newcommand{\ith}{\mathit{th}}

\newcounter{jmpnumber}

\newcommand{\ds}{\displaystyle}

\def\cov{\mathbb{C}\mathrm{ov}}

\newcommand{\be}{\begin{equation}}
\newcommand{\ee}{\end{equation}}
      \newtheorem{theorem}{Theorem}[section]
       \newtheorem{proposition}[theorem]{Proposition}
       \newtheorem{corollary}[theorem]{Corollary}
       \newtheorem{lemma}[theorem]{Lemma}

\newtheorem{definition}[theorem]{Definition}

\newtheorem{notation}[theorem]{Notation}

\newtheorem{example}{Example}[section]
\theoremstyle{remark}
\newtheorem{remark}[theorem]{Remark}

\title[Asymptotic $\ast$--distribution of permuted Haar unitary matrices]
{Asymptotic $\ast$--distribution of permuted\\[5pt] Haar unitary matrices}

\author[J.~A.~Mingo]{James A. Mingo} 
\address[J.~A.~Mingo]{Department
	of Mathematics and Statistics, Queen's University, Jeffery
	Hall, Kingston, Ontario, K7L 3N6, Canada}
\email{mingo@mast.queensu.ca} 
\thanks{J.A.M.: Research supported by a Discovery Grant from the
Natural Sciences and Engineering Research Council of Canada.}

\author[M. Popa]{Mihai Popa}
\address[M. Popa]{Department of Mathematics\\ University of Texas at San Antonio\\ One UTSA Circle San Antonio\\ Texas 78249, USA
	\\ and “Simon Stoilow” Institute of Mathematics of the Romanian Academy\\ P.O. Box 1-764\\ 014700 Bucharest, Romania}
\email{mihai.popa@utsa.edu}
\thanks{M.P.: research partially supported by the Simons
  Foundation grant No. 360242.}

\author[K. Szpojankowski]{Kamil Szpojankowski}
\address[K. Szpojankowski]{
Wydzia\l{} Matematyki i Nauk Informacyjnych\\
Politechnika Warszawska\\
ul. Koszykowa 75\\
00-662 Warszawa, Poland.}
\email{k.szpojankowski@mini.pw.edu.pl}
\thanks{K.Sz.: research partially supported by NCN grant 2016/23/D/ST1/01077}

\begin{document}
	\begin{abstract}
		We study Haar unitary random matrices with permuted entries. For a sequence of permutations $\left(\sigma_N\right)_N$, where $\sigma_N$ acts on $N\times N$ matrices we identify conditions under which the $\ast$--distribution of permuted Haar unitary matrices $U_N^{\sigma_N}$ is asymptotically circular and free from the unpermuted sequence $U_N$. We show that this convergence takes place in the almost sure sense. Moreover we show that our conditions on the sequence of permutations are generic in the sense that are almost surely satisfied by a sequence of random permutations.
	\end{abstract}
	\maketitle
\section{Introduction}

Given an $N\times N$ matrix $A_N$ and a bijection $\sigma_N: \{1,\ldots,N\}^2\to\{1,\ldots,N\}^2$ we define a permuted matrix $A_N^{\sigma_N}$ by $\left[A^{\sigma_N}\right]_{i,j}=\left[A\right]_{\sigma_N(i,j)}$. In this paper we study in detail permuted Haar unitary random matrices. We are mostly interested in asymptotic properties of permuted Haar unitary random matrices form the point of view of non--commutative probability and in particular free probability. 

Random matrices with randomly permuted entries were studied in research literature, in particular in the context of limiting properties of spectral measure. In \cite{Chatterjee} the author studies self-adjoint matrices for which the collection of upper triangular entries forms an exchangeable family, and shows that for such random matrices, after suitable normalization, the  empirical distribution of eigenvalues converges weakly in probability to Wigner's semicircle law. Of course if we take a uniformly chosen random permutation of any family of random variables, it becomes exchangeable. Thus results from \cite{Chatterjee} can be thought of as random permutations of entries of selfadjoint random matrices, which keep the matrix Hermitian.

A non-Hermitian version of results from \cite{Chatterjee} were studied in \cite{AdamczakChafaiWolff}, where the authors showed, under some technical assumptions, that if entries of a random matrix form an exchangeable family, then after proper normalization the empirical spectral distribution converges weakly in probability to the circular law. Again one can interpret this result in terms of random permutations, where now in contrast to \cite{Chatterjee} we take a uniformly chosen permutation of all entries.

Asymptotic distributions of large random matrices are closely related with free probability. It is well known in that large, Hermitian, unitarily invariant, independent random matrices (under some technical assumptions) become asymptotically free (c.f. \cite{VoiculescuLimit} and \cite{MingoSpeicher} for details). As conjugation with a unitary matrix corresponds to a change of basis, this could be seen as a statement about matrices with randomly chosen basis of eigenvectors. In this paper we continue another thread where asymptotic freeness emerges in a surprising way, namely we look at random matrices with permuted entries.

Entry permutations of matrices have been relevant in various areas. The matrix transpose is the standard example of positive but not completely positive maps (see \cite{paulsen}), partial transposes (see \cite{HorHorHor, AubrSzaWer,ArizNechVarg}) and the so-called mixing map (see \cite{billiard,MandLinowZycz}) are used in quantum information theory.
The connection between entry permutations and free independence was first established in \cite{MingoPopaUnitarlyInvariant2016}, where it is shown that unitarily invariant random matrices are asymptotically free from their transposes. Later works address the connection to free independence in the setting of partial transposes and Wishart random matrices (see \cite{MingoPopaWishart1} and \cite{MingoPopaWishart2}) or more general permutations and matrices with non-commutative entries (see \cite{PopaHao1}, \cite{PopaHao2}). 
In the present paper instead of focusing on particular type of permutation of entries we study a general class of permutations of entries for Haar Unitary matrices. Simpler framework of Gaussian random matrices was considered in \cite{PopaGEPAF}. We also show random sequences of permutations almost surely belong to the class of permutations which we consider.

More precisely in this paper we focus on permutations of Haar unitary random matrices, that is we consider the normalized Haar measure on the group of $N\times N$ unitary matrices. For each $N$ we take a random Haar unitary matrix $U_N$ and a permutation $\sigma_N$ of the entries, then  we consider the sequence of matrices $\left(U_N^{\sigma_N}\right)_{N\geq 1}$. We do not study the limiting spectral distribution as in \cite{Chatterjee,AdamczakChafaiWolff}, but we are interested in the limiting, as $N\to\infty$, $\ast$--distribution of $\left(U_N^{\sigma_N}\right)_{N\geq 1}$ and limiting joint distribution of the pair $\left(U_N,U_N^{\sigma_N}\right)_{N\geq 1}$, seen as a non--commutative variables in some non--commutative probability space. We review basic notions of non--commutative probability, such as $\ast$--distribution and joint distribution, in Section 2. We identify a sufficient condition (C) for the non--random sequence of permutations (we refer to Notation \ref{not:41} and Definition \ref{def:permprop}, to see the precise formulations) which allows us to prove the following result.
\begin{theorem}\label{thm:11}
	If $(\sigma_N)_N$ satisfies \emph{ (C) } then $U_N^{\sigma_N}$ converges in $\ast$--moments to a circular element and  $\left(U_N\right)_{N\geq 1}$ and $\left(U_N^{\sigma_N}\right)_{N\geq 1}$ are asymptotically $*$-free.
\end{theorem}
The results above say exactly the following, for any polynomial $Q\in\mathbb{C}\langle x,x^*,y,y^*\rangle$ in the non-commuting variables $x, x^*, y, y^*$ we have
\begin{align}\label{eq:asymptotic_freeness}
\lim_{N\to\infty} \frac{1}{N} 
\E \big(\Tr\big(Q\big(U_N^{\sigma_N},(U_N^{\sigma_N})^*,U_N,U_N^{*}\big)\big)&=\varphi\big(Q(c,c^*,u,u^*)\big),
\end{align}
where $c$ is a circular element, $u$ is Haar unitary, and moreover $c$ and $u$ are $*$-free. We review circular, Haar unitary, and more general $R$--diagonal elements in Remark \ref{rem:Rdiag}. The fact that the $*$-distribution of $U_N$ converges to a Haar unitary is well known (see for example \cite[Ch. 4]{HiaiPetz}).

Moreover we show that the  convergence above is in fact not only in expectation but also with probability one. Thus we prove the following result.
\begin{theorem}\label{thm:12}
	If $(\sigma_N)_N$ is a given sequence of entry permutations that satisfies \emph{(C)}, then $U_N^{\sigma_N}$ converges almost surely in $\ast$-moments to a circular element and  $\left(U_N\right)_{N\geq 1}$ and $\left(U_N^{\sigma_N}\right)_{N\geq 1}$ are almost surely asymptotically  $*$-free.
\end{theorem}\noindent
Thus Theorem \ref{thm:12} says that the limit in (\ref{eq:asymptotic_freeness})   can be replaced by almost sure convergence
\begin{align*}
\lim_{N\to\infty} \frac{1}{N} 
\Tr \big(Q\big(U_N^{\sigma_N},(U_N^{\sigma_N})^*,U_N,U_N^{*}\big)\big)&=\varphi\big(Q(c,c^*,u,u^*)\big).
\end{align*}

In the subsequent sections we work with two more general conditions (C1) and (C2), where property (C1) assures that the asymptotic $*$--distribution of the permuted matrix is circular and (C2) gives $*$-freeness between the permuted and the unpermuted matrix. We show that if a sequence of permutations satisfies (C) then it automatically satisfies both (C1) and (C2).
We give the precise definitions of these conditions in Section \ref{sec:asymptotic_distribution}. Roughly speaking,  condition (C) says that there is a relatively small number of pairs of entries which were in the same row or column, and after permutation $\sigma_N$ still are in the same row or column. Heuristically, this condition forces the permutation $\sigma_N$ to break the structure of a unitary matrix.

Next we prove that condition (C) is satisfied almost surely by sequences of uniformly chosen random permutations. This result shows that condition ($C$) is rather generic. It also shows also that conditions (C1) and (C2) are satisfied almost surely by sequences of random permutations. This observation together with the results above can be summarized in the following theorem.
\begin{theorem}\label{thm:13}
	If $(\sigma_N)_N$ is a sequence of random permutations, chosen uniformly, and independent from the sequence  $(U_N)_N$, then $(U_N^{\sigma_N})$ converges almost surely (with respect to $ ( \sigma_N)_N $) in $\ast$--moments to a circular element and $\left(U_N\right)_{N\geq 1}$ and $\left(U_N^{\sigma_N}\right)_{N\geq 1}$ are almost surely asymptotically  $*$-free.
\end{theorem}

In the subsequent sections we consider in fact the more general framework of a tuple of permutations $(\sigma_{N,1},\ldots,\sigma_{N,k})$. We identify a condition (C3) (See Definition \ref{def:pairperm}) which concerns pairs of permutations which allows us to state the following result.

\begin{theorem}\label{thm:14}
	Let $(\sigma_{N,1},\ldots,\sigma_{N,k})$ be permutations on $[N]\times[N]$. Suppose that for each $i$ permutation $\sigma_{N,i}$ satisfies condition \emph{(C)} and for each $i\not= j$ the pair $\{\sigma_{N,i},\sigma_{N,j}\}$ satisfies condition \emph{(C3)} then the tuple 
	$(U_N,U_N^{\sigma_{N,1}},\ldots,,U_N^{\sigma_{N,k}})$ converges in $*$--moments to the tuple $(u,c_1,\ldots,c_k)$, where $u$ is a Haar unitary $*$-free from $\{c_1,\ldots,c_k\}$ which is a family of free circular elements.
\end{theorem}

We also show that the result above holds in the almost sure sense and observe that condition (C3) is satisfied with probability 1 by two independent sequences of random permutations.

Let us explain here the relation of our work with Adamczak, Chafa\"i, and Wolff
\cite{AdamczakChafaiWolff}. In \cite{AdamczakChafaiWolff} the authors consider random permutations of some general random matrices with real entries and look at the limiting spectral distribution. We are concerned only with Haar unitary random matrices (thus with complex entries), and study the asymptotic behaviour of $\ast$--moments. Of course in view of results from \cite{AdamczakChafaiWolff} it is expected that the asymptotic $*$--distribution will be the one of circular element, moreover one has that the Brown measure of the circular element equals the circular law (see \cite{HaagerupLarsen}). However spectral convergence is not continuous in the topology of convergence of $\ast$--moments (for a discussion about this we refer to \cite{SniadyBrownReg}). Moreover the results of \cite{AdamczakChafaiWolff} apply to sequences of uniformly random  permutations, while we identify conditions for non--random sequences, which are satisfied almost surely by random sequences of permutations.

We take advantage of the explicit form of our conditions and we identify several classes of permutations which satisfy conditions (C1) and (C2) and are inspired by ones appearing in the quantum information theory literature. Thus we believe that our result will be of some interest for the quantum information community. In Example \ref{thm:partialtranspose} we show that partial transposes with growing size of blocks and at the same time growing number of blocks satisfies our conditions. We show (see Example \ref{thm:mixingmap}) that the mixing map from \cite{MandLinowZycz} and \cite{billiard} is another example of a sequence of permutations which satisfies our conditions. A different situation,
addressed in Example \ref{d-part-transp}, is the case of partial transposes with fixed number of blocks, equal say $n=d^2$, for some $d>0$. This is an example of a sequence which does not satisfy property (C1); we determine its asymptotic $\ast$--distribution as $\tfrac{1}{d}\left(u_1+\ldots+u_{d^2}\right)$, where $u_1,\ldots,u_{d^2}$ are free Haar unitaries. Yet, we still have asymptotic free independence from the initial (un-permuted) Haar unitary. Table \ref{table:1} summarizes results of Section 5, it contains a list of permutations together with an indication if a given permutation satisfies (C1) and (C2).

\begin{table}[h!]
	\centering
	\begin{tabular}{|c| c| c|}
		\hline
		Entry permutation & (C1) asymptotic  & (C2) asymptotic freeness \\ [0.5ex] &circular distribution &  from unpermuted matrix \\
		\hline
		Identity & $\Cross$ & $\Cross$ \\ [0.5ex]
		Transpose & $\Cross$  & \checkmark \\ [0.5ex]
		Partial transpose  &  &  \\
		(fixed number of blocks) & $\Cross$ & \checkmark \\[0.5ex]
		Partial transpose &  &  \\
		(number of blocks and size& \checkmark & \checkmark \\
		of a block grows to $\infty$) &  &  \\[0.5ex]
		Mixing map & \checkmark & \checkmark \\[1ex]
		\hline
	\end{tabular}
	\caption{List of entry permutations together with information if they satisfy (C1) and (C2).}
	\label{table:1}
\end{table}

In the setting of Wishart matrices, partial transposes with a fixed number of blocks also behave distinctly from the partial transposes with both the number of blocks and their size growing to infinity. The limit distribution is a rescaled difference of two free Marchenko-Pastur laws in contrast to a shifted semicircular (see \cite{BaNec}, \cite{MingoPopaWishart1}). Moreover there is no asymptotic free independence from the initial Wishart ensemble,
as well as from the ones permuted via partial transposes with growing both size and number of blocks, as shown in \cite{MingoPopaWishart1}. 

Besides this Introduction this paper has 6 more sections. In Section 2 we introduce necessary notions from free probability and discuss the Weingarten calculus from \cite{CollinsSniady}. Section 3 is devoted to the study of asymptotic behaviour of some integrals on the unitary group. In Section 4 we prove Theorem \ref{thm:11}. In Section 5 we discuss examples of deterministic permutations which satisfy our assumptions (C1) and (C2) and the example of partial transpose with fixed number of blocks. Section 6 is devoted for the proof of Theorem \ref{thm:12}. In Section 7 we show that conditions (C) and (C3) are satisfied almost surely by sequences of random permutations, thus proving Theorem \ref{thm:13}.

\section{Background and notation}\label{section:background}
In this section we introduce briefly the main tools which we use in the present paper. First we focus  on non-crossing partitions and their role in free probability. Next we discuss the Weingarten calculus together with a reformulation which will be necessary in subsequent sections.
\subsection{Non-crossing partitions and free cumulants}
\begin{definition}\label{def:pairings_partitions} $\ $
 	
	\begin{enumerate} 
		\item	For a positive integer 
		$n$ 
		we denote by
		 $[n] $ the ordered set 
		 $\{1,\ldots,n\}$. 
		 A partition $\pi$ of $[n]$ is a collection of non-empty, pair-wise disjoint subsets  $B_1,\ldots,B_k\subseteq [n]$ such that $\bigcup_{j=1}^k B_j=[n]$. The subsets $B_j$ for $j=1,\ldots,k$ are called blocks of $\pi$, the number of blocks in $\pi$ is called the size of $\pi$ and is denoted by $|\pi|$, i.e. we have $|\pi|=k$.  
		
		\noindent
		The family of all partitions of  $[n]$ is denoted by  
		${P}(n)$. 	
		\item We say that
		 $\pi\in {P}(n)$ 
		 is a non-crossing partition if whenever $B_1,B_2$ are blocks of $ \pi $
		 such that 
		 $ i_1, i_2 \in B_1 $ 
		 and 
		 $ j_1, j_2 \in B_2 $
		 for some 
		 $ i_1<j_1<i_2<j_2\leq n$,
		 then
		 $ B_1 = B_2$.
		   The family of all non-crossing partitions of $[n]$ is denoted by $NC(n)$. It is well known that the number of non-crossing partitions in $NC(n)$ is equal to the $n^\ith$ Catalan number $C_n=\tfrac{1}{n+1}\binom{2n}{n}$, for $n\geq 1$. The set of non-crossing partitions on $[2n]$ with the property that every block has exactly two elements is denoted by $NC_2(2n)$.
		\item If every block of a partition
		 $\pi\in {P}(2n)$
		  has exactly two elements, we call such partition a pairing and we denote by ${P}_2(2n)$
		 the set of all pair partitions on $\{1,\ldots,2n\}$. Sometimes we will identify 
		$\pi\in {P}_2(2n)$ 
		with a permutation in $S_{2n}$, where each block of $\pi$ becomes a cycle of the permutation. Such permutation has the property that $\pi$ has no fixed point and $\pi^2$ is the identity permutation in $S_{2n}$.
		\end{enumerate}
\end{definition}

Next we recall the definitions of free cumulant functionals related to non-crossing partitions which play a very important role in free probability (cf. \cite{NicaSpeicherLect}, Lectures 9 and 10).
\begin{definition}
	For every
	 $n\geq 1$
	  free cumulant functional
	   $\kappa_n:\mathcal{A}^n\to \mathbb{C}$
	    is defined  recursively through equations $($valid for all 
	    $ m \geq 1 $ and 
	    all 
	    $ m $-tuples
	    $ (a_1, a_2, \dots, a_m ) \in \mathcal{A}^m)$:
	\begin{align*}
 \varphi(a_{1}\cdots a_{m})=\sum_{\pi\in NC(m)}\,\kappa_\pi(a_{1},\ldots,a_{m}),
	\end{align*}
	where for 
	$\pi=\{B_1,\ldots,B_k\}\in NC(m)$
	we denote
	\begin{align*}
	\kappa_\pi(a_1,\ldots,a_m)=\prod_{j=1}^k\,
	\kappa_{|B_j|}\left(a_i;i\in B_j\right).
	\end{align*}
\end{definition}

It turns out that freeness can be characterized in terms of free cumulants as
follows \cite[Theorem~11.16]{NicaSpeicherLect}:
random variables $X_1$ and $X_2$ are free
if  $\kappa_n\left(X_{i_1}, X_{i_2}, \ldots, X_{i_n} \right)=0$ whenever $n\geq2$,
and there are $1\leq l, k\leq n$ such that $i_l\neq i_k$.

In this paper we will deal mainly with two types of non-selfadjoint random variables: Haar unitary and circular elements. They are both examples of a more general and important family of so called $R$--diagonal elements, introduced in \cite{NicaSpeicherRdiag}. 

\begin{remark} $\ $ \label{rem:Rdiag}
	
	\begin{enumerate}
		\item Fix a non-commutative probability space $(\mathcal{A},\varphi)$ and consider related free cumulants functionals $(\kappa_n)_{n\geq 1}$. We say that a non-selfadjoint element $a\in\mathcal{A}$ is an $R$--diagonal element if for $\kappa_n(a^{\epsilon_1},\ldots,a^{\epsilon_n})\neq 0$ only when $n$ is even and $\epsilon_1\neq \epsilon_2\neq\ldots\neq\epsilon_n$. In other words only even length, alternating cumulants in $\{a,a^*\}$ are non-zero.
		
		\item 
We will consider only non--commutative probability spaces where $\varphi$ is trace. Then $\kappa_{2n}(a^*, a, \ldots ,a^*,a) = \kappa_{2n}(a, a^*, \ldots ,a,a^*)$. For an $R$-diagonal element $a$, the sequence $(\kappa_{2n}(a^*, a, \ldots ,a^*,a))_{n\geq 1}$ is called the \textit{determining sequence} of $a$.
		
		\item Two canonical examples of $R$-diagonal elements are  circular elements and Haar unitaries. By a \textit{circular element} we mean an operator $c$, such that $\kappa_2(c,c^*)=\kappa_2(c^*,c)=1$ and all other $*$-cumulants vanish. A \textit{Haar unitary} is a unitary  i.e. $uu^*=u^*u=1$ such that $\varphi(u^k)=0$ for all $k\in \mathbb{Z}\setminus \{0\}$. The determining sequence for a Haar unitary element is given by $\kappa_{2n}(u^*, u, \dots, u^*, u) = (-1)^{n-1}C_{n-1}$ where $C_n = \frac{1}{n+1} \binom{2n}{n}$ is the $n^\ith$ Catalan number.
	\end{enumerate}
\end{remark}

\subsection{Unitary Weingarten calculus}
We first recall the definition and properties of the unitary Weingarten function, $\Wg_N$,  we need. Recall from \cite{CollinsSniady} that the Weingarten function uses Schur-Weyl duality to reduce an integral over the unitary group to a sum over the symmetric group. 
\begin{remark}
	If 
	$U=[u_{i,j}]_{1\leq i,j\leq N}$ 
	is an $N\times N$
	 Haar distributed unitary random matrix then
	\begin{align}\label{eqn:Weingarten}
	\E\big(
	u_{i_1,j_1}u_{i_2,j_2}\cdots u_{i_n,j_n}
	\overline{u_{i^\prime_1,j^\prime_1}}\,  &
	\overline{u_{i^\prime_2,j^\prime_2}}
	\cdots
	\overline{u_{i^\prime_n,j^\prime_n}}
	\big)\\
\nonumber	= &
	\sum_{\sigma,\tau\in S_n}
	\big(
	\prod_{k=1}^n
	\delta_{i_k,i^\prime_{\sigma(k)}}
	\delta_{j_k,j^\prime_{\tau(k)}}
	\big) 
	\mathrm{Wg}_N (\sigma^{-1}\tau),
	\end{align}
Where $\mathrm{Wg}:S_n\times \mathbb{N}\to \mathbb{R}$ is the unitary Weingarten function. $\Wg_N(\sigma)$ is a rational function of $N$ with integer coefficients.

We will not need the exact values of the Weingarten function, we will be interested the leading term when expanded as power series in $\frac{1}{N}$. For $\sigma\in S_n$, let $\#(\sigma)$ be the number of cycles in the cycle decomposition of $\sigma$. Then the first non-zero term in the $\frac{1}{N}$ expansion of $\Wg_N(\sigma)$ is $N^{\#(\sigma)-2n}$. The Weingarten function is a central function on $S_n$, this means that its value only depends on the cycle decomposition of a permutation. Suppose $\sigma$ has the cycle decomposition $c_1 \cdots c_k$ we let $|c_i|$ be the number of elements in the $i^{th}$ cycle $c_i$.
Then
\begin{align*}
\mathrm{Wg}_N (\sigma)
=N^{\#(\sigma)-2n+} \prod_{i=1}^{k} (-1)^{|c_i|-1} C_{|c_i|-1}+O(N^{\#(\sigma) - 2(n+1)})
\end{align*}
where $C_l = \frac{1}{2l +1} \binom{2l}{l}$ is the $l^{\ith}$ Catalan number. Note the appearance of the cumulants of a Haar unitary in Remark \ref{rem:Rdiag} (2). $\hfill $ $ \square $
\end{remark}

For our purposes it is more convenient to keep track of which indices in (\ref{eqn:Weingarten}) of the $u$'s are paired with indices of the $\overline{u}$'s using pairings, instead of permutations,. In the following remark we explain how one can do this. Recall from Definition \ref{def:pairings_partitions} (3) that we will think of a pairing as a permutation with all cycles of length 2. 

\begin{remark}\label{remark:2.5}
Consider the set
 ${P}_2^\delta(2n) = \{p \in {P}_2(2n) \mid p(t) > n \mbox{ for } t \in [n] \}$;
  i.e. for each pair
   $(r, s)$
    in $p$ one element, $r$ say, is in $[n]$ and the other, in this case $s$,  is in 
     $[n+1, 2n]\ab = \{n +1, \dots, 2n\}$. 
     Define $\Psi:S_n\to {P}_2^\delta(2n)$ 
     by
      $\Psi(\sigma)(t)=n+\sigma(t)$ for $t \in [n]$. It is immediate to see that $\Psi$ is a bijection. Indeed, there is a simple algebraic relation  that will be useful. We consider $S_n$ to be a subgroup of $S_{2n}$ in the usual way; namely a permutation in $S_n$ acts trivially on the set $[n+1, 2n]$. We let $\delta \in S_{2n}$ be the permutation given in cycle form by $(1, n+1) (2, n +2) \cdots \ab (n, 2n)$. Then $\Psi(\sigma) = \sigma^{-1} \delta \sigma$. Note that if $\gamma = (1,2,3, \dots, n) \in S_n$ then $\delta \gamma$ is the permutation with the single cycle $(\delta(1), 1, \delta(2), 2, \dots, \delta(n), n)$. If $\sigma \in NC(n)$ i.e. $\sigma$ is non-crossing with respect to $\gamma$ then, $\Psi(\sigma)$ is non-crossing with respect to $\delta \gamma$. 

	Consider now permutations 
	$\sigma,\tau\in S_n$,
	 let
	  $p = \Psi(\tau)$
	   and 
	   $q = \Psi(\sigma)$
	    and consider the permutation
	     $\sigma^{-1}\tau$.
Observe that as permutations, $pq$ maps $[n]$ to $[n]$ and the restriction of $pq$ to $[n]$ is $\sigma^{-1}\tau$. In \cite[Lemma 2]{MingoPopaOrth} it was shown that the cycle decomposition of $pq$ can be written as $c_1 c_1 ' \cdots c_k c'_k$ with $c'_i = qc^{-1}q$. If we view $p$ and $q$ as partitions then the blocks of $p \vee q$ are $\{ c_1 \cup c_1', \dots, c_k \cup c'_k\}$.  Hence for each cycle
	        $(l_1,\ldots,l_t)$ in $\sigma^{-1}\tau$,
	         $\{l_1,p(l_1),l_2,p(l_2)\ldots, l_t, p(l_t)\}$ 
	         will be a block
	          $p \vee q$.
    Observe that all blocks of
     $p \vee q$ 
     will have an even number of elements and
      thus 
       $p \vee q$
        is the partition of
        $[2n]$
         with blocks
          $\{a_1,\ldots,a_{2m}\}$ (for some $m \in [n]$)
             of the form 
	\begin{align}
	\label{eqn:partition}
	a_2=p(a_1),\,a_3=q(a_2),\,a_4=p(a_2),\ldots,a_1=q(a_m).
	\end{align}
	For example take 
	$\sigma=(1)(2,3)$ 
	and
	 $\tau=(1,2,3)$ 
	 then
	  $\sigma^{-1}\tau=(1,3)(2)$
	   and we have
	    $p \vee q=\{( 1,3,4,5),(2,6)\}$ 

An important property of the Weingarten function is that $\Wg_N(\sigma)$ depends only on the conjugacy class of $\sigma$. So for any two pairings 
$ p, q \in \cP_2(2n) $   we define
 $\Wg_N(p,q)$ as follows. We write the cycle decomposition of the product $pq$ as $c_1 c'_1 \cdots c_k c'_k$, and then define $\Wg_N(p,q)$ to be $\Wg_N(\sigma)$ where $\sigma \in S_n$ is any permutation with the cycle type of $c_1 \cdots c_k$. $\Wg_N(p,q)$
	  has the asymptotic behaviour
	\begin{align*}
	\Wg_N(p,q)=N^{-2n+ |p \vee q|}\prod_{i=1}^{|p \vee q|} (-1)^{|B_i|/2-1} C_{|B_i|/2-1}+O(N^{-2n+ | p \vee q |-2}),
	\end{align*}
	where $p \vee q=\{B_1,\ldots,B_{k}\}$. We shall let $C_{p,q}$ denote the coefficient of $N^{-2n + \#(p \vee q)}$ above, namely $C_{p,q} = \ds \prod_{i=1}^{|p \vee q|} (-1)^{|B_i|/2-1} C_{|B_i|/2-1}$.

Given 
$p$, $q \in \cP_2^\delta(2n)$
 with
  $\sigma, \tau \in S_n$ 
  such that $\Psi(\sigma) = p$ and $\Psi(\tau) =q$, we have $pq|_{[n]} = \sigma^{-1} \tau$ and the pairs of the cycles of $pq$ are such that one is in $[n]$ and the other is in $[n+1, 2n]$. Thus according to our definition we have $\Wg_N(p,q) = \Wg_N(\sigma^{-1} \tau)$. Then Equation (\ref{eqn:Weingarten}) can be rewritten as
	\begin{align}\label{eqn:Weingarten2}
	\E\big(
	u_{r_1,l_1}u_{r_2,l_2}
	\cdots 
	u_{r_{n},l_{n}} \overline{u_{r_{n+1},l_{n+1}}}
	\cdots
	\overline{u_{r_{2n},l_{2n}}}
	\big)
	=
	\sum_{p,q\in \mathcal{P}_2^\delta(2n)}
	\big(
	\prod_{k=1}^n\delta_{r_k,r_{p(k)}}\delta_{l_k,l_{q(k)}}
	\big)
	 \Wg_N(p,q).
	\end{align}\qed\end{remark}

	We next need a reformulation of Equation \eqref{eqn:Weingarten2} in which elements of the matrix $U$  are not ordered so that all complex conjugates are grouped to the right. In principle this always can be done, as all the random variables commute, however to have good control of the formulas we carefully introduce here how this can be done.
\begin{remark}\label{rem:24}
	We are interested in formula for 
	$		\E \big(
		u_{i_1,j_1}^{\epsilon(1)}
		u_{i_2,j_2}^{\epsilon(2)}
		\cdots u_{i_{2n},j_{2n}}^{\epsilon(2n)}
		\big),
$
		where 
	$\epsilon\in\{1,*\}^n$.
	We denote
	 $u_{i,j}^*=\overline{u_{i,j}}$
	  and 
	  $u_{i,j}^1=u_{i,j}$.
	   In view of \eqref{eqn:Weingarten} we are only interested in such sequences $\epsilon$
	    for which $1$ and $*$ each appear $n$ times. 
	Define
	\begin{align*}
	\rho_\epsilon(s)=\begin{cases}
		|\{t\leq s \mid \epsilon(t)=1\}| & \mbox{ if } \epsilon(s)=1,\\
		n+|\{t\leq s \mid \epsilon(t)=*\}| & \mbox{ if } \epsilon(s)=*.
	\end{cases}
	\end{align*}
If we write $\epsilon^{-1}(1) = \{r_1, \dots, r_n\}$ with $r_1 < \cdots < r_n$ and $\epsilon^{-1}(*) = \{s_1, \dots, s_n\}$ with $s_1 < \cdots < s_n$ then $\rho_\epsilon(r_k) = k$ and $\rho_\epsilon(s_k) = n + k$ for $1 \leq k \leq n$. 
Let $k(t) = i(\rho_\epsilon^{-1}(r_t))$ and $l(t) = i(\rho_\epsilon^{-1}(s_t))$ then
\begin{align}\label{eq:re-indexing}
\E \big(
		u_{i_1,j_1}^{\epsilon(1)}
		u_{i_2,j_2}^{\epsilon(2)}
		\cdots u_{i_{2n},j_{2n}}^{\epsilon(2n)}
		\big) = 
		\E\left(u_{k_1,l_1}u_{k_2,l_2}\ldots u_{k_{n},l_{n}} \overline{u_{k_{n+1},l_{n+1}}}\ldots\overline{u_{k_{2n},l_{2n}}}\right).
	\end{align}
Now let
 $ \cP_2^\epsilon(2n) = \{ p \in \cP_2(2n) \mid \epsilon_u \not = \epsilon_v$
  for all pairs $(u, v) \in p\}$. Then $p \mapsto \rho_\epsilon^{-1} p \rho_\epsilon$ is a bijection from $ \cP_2^\epsilon(2n)$
  to 
  $ \cP_2^\delta(2n)$. 
  Let
   $\tilde p = \rho_\epsilon^{-1} p \rho_\epsilon$, then 
   $i = i \circ \tilde p$ $\Leftrightarrow$  	$i = i \circ \rho_\epsilon^{-1} \circ p \circ \rho_\epsilon$ $\Leftrightarrow$ $i \circ \rho_\epsilon^{-1} = i \circ \rho_\epsilon^{-1} \circ p$ $\Leftrightarrow$ $k = k\circ p$.
    Thus
\begin{equation}\label{eq:subscripts}
\prod_{(\tilde u, \tilde v) \in \tilde p} \delta_{i_{\tilde u}, i_{\tilde v}} =
\prod_{(u,v) \in p} \delta_{k_u, k_v}
\mbox{ and }
\prod_{(\tilde u, \tilde v) \in \tilde q} \delta_{j_{\tilde u}, j_{\tilde v}} =
\prod_{(u,v) \in q} \delta_{l_u, l_v}.
\end{equation}
We can make our notation even more compact by writing $\delta_{i, i\circ  p} $ for $\prod_{(u, v) \in p} \delta_{i_{u}, i_{v}}$. 
From our definition of $\Wg_N(p,q)$ we have that $\Wg_N(p,q) = \Wg_N(\rho_\epsilon^{-1} p \rho_\epsilon, \rho_\epsilon^{-1} q \rho_\epsilon)$. With Equations (\ref{eq:re-indexing}), (\ref{eq:subscripts}), and the relation above we have that Equation (\ref{eqn:Weingarten2}) can now be written as
\begin{equation}\label{eq:weingarten_3}
\E \big(
		u_{i_1,j_1}^{\epsilon(1)}
		u_{i_2,j_2}^{\epsilon(2)}
		\cdots u_{i_{2n},j_{2n}}^{\epsilon(2n)}
		\big) =
\sum_{p, q \in \cP_2^\epsilon(2n)} 
 \delta_{i, i \circ p} \delta_{j, j \circ q} 
\Wg_N(p,q).
\end{equation}
\end{remark}

\section{Integrals of permuted Haar unitary random matrices}

Let $U_N$ be an $N\times N$ random Haar unitary matrix. As we mentioned in the introduction we will work with its permuted copies. Let us fix 
$N,n\in\mathbb{N}$,
 and fix 
 $\sigma_{1,N},\ldots,\sigma_{2n,N}\in S\left([N]^2\right)$, then for $k=1,2,\ldots,2n$ matrix $U_N^{\sigma_{k,N}}$ is defined as $\left[U_N^{\sigma_{k,N}}\right]_{i,j}=\left[U_N\right]_{\sigma_{k,N}(i,j)}$.
 In order to make the formulas more transparent we suppress $N$ in most of the notation we use, and write 
$\sigma_k $
 instead of 
 $\sigma_{k,N}$,
  for 
  $ k =1,\ldots,k$. 
  Similarly we shall write
   $U$ for $U_N$. One should keep in mind that we work with several permutations of the same matrix $U_N$.

Our main goal in this section is to understand the asymptotic behaviour of integrals of the type
\begin{align*}
\E\circ \tr
\big(
U^{(\sigma_1,\varepsilon_1)}
U^{(\sigma_2,\varepsilon_2)}
\cdots 
U^{(\sigma_{2n},\varepsilon_{2n})}
\big),
\end{align*}
where
 $\varepsilon_s\in\{1,\ast\}$,
 i.e.
 $ \varepsilon_s = \varepsilon(s) $
  for some map 
 $ \varepsilon : \{ 1, 2, \dots, 2n \} \rightarrow \{ 1, \ast \} $,
  and
\begin{align*}
U^{(\sigma_s,\varepsilon_s)}=\begin{cases}
U^{\sigma_s} & \mbox{ if } \varepsilon_s=1,\\
\big(U^{\sigma_s}\big)^\ast &\mbox{ if } \varepsilon_s=\ast.
\end{cases}
\end{align*}
First we will rewrite the above integral in a form which allows us to use the Weingarten calculus.

According to the review of the Weingarten calculus in the previous section, all integrals with odd number of products are zero, similarly we need to have equal number of  $\ast$'s and $1$'s.

Expanding the trace we get
\begin{align}\label{eqn:31}
\E\circ \tr
\big(
U^{(\sigma_1,\varepsilon_1)}U^{(\sigma_2,\varepsilon_2)}
\cdots
 U^{(\sigma_n,\varepsilon_{2n})}
\big)=
\sum_{i_1,\ldots,i_{2n}=1}^N\frac{1}{N}
\E\big(
\big[U^{(\sigma_1,\varepsilon_1)}\big]_{i_1,i_2}
\cdots
\big[U^{(\sigma_{2n},\varepsilon_{2n})}\big]_{i_{2n},i_1}
\big),
\end{align}
where
\begin{align*}
\big[
U^{(\sigma_s,\varepsilon_s)}
\big]_{k,l}
=\begin{cases}
& \left[U^{\sigma_s}\right]_{k,l}=u_{\sigma_s (k,l)} \mbox{ if } \varepsilon_s =1\\
&\\
&\overline{\left[U^{ \sigma_s } \right]}_{l,k}=\overline{u_{\sigma_s (l,k)} } \mbox{ if }
 \varepsilon_s = \ast,
\end{cases}
\end{align*}
for 
$ s = 1, \dots, 2n $.

We need to introduce another bit of notation, for $ s = 1, \dots, 2n $ we will write
\begin{align*}
\big[
   U^{(\sigma_s ,\varepsilon_s)} 
   \big]_{k,l}
   =u^{(\varepsilon_s)}_{\sigma_s \circ \varepsilon_s (k,l)},
\end{align*}
where
\begin{align*}
a^{(\varepsilon_s )}=\begin{cases}
&a \mbox{ if } \varepsilon_s =1,\\
&\overline{a} \mbox{ if } \varepsilon_s =\ast\\
\end{cases}
\end{align*}
and 
\begin{align*}
\varepsilon_s (i,j)=\begin{cases}
&(i,j) \mbox{ if } \varepsilon_s =1,\\
&(j,i) \mbox{ if } \varepsilon_s =\ast.\\
\end{cases}
\end{align*}
Continuing the equation \eqref{eqn:31} we get
\begin{align}\label{eqn:32}
\sum_{i_1,\ldots,i_{2n}=1}^N
\frac{1}{N}
\E\big( &
u^{(\varepsilon_1)}_{\sigma_1\circ\varepsilon_1(i_1,i_2)} 
\cdots u^{(\varepsilon_{2n})}_{\sigma_{2n}\circ\varepsilon_{2n}(i_{2n},i_1)}
\big)\\&=
\sum_{\substack{i_1,\ldots,i_{2n}\in\{1,\ldots,N\},\\
				j_1,\ldots,j_{2n}\in\{1,\ldots,N\},\\
				j_k=i_{k+1} \textrm{ for } k=1,\ldots,{2n}-1,\\
				j_{2n}=i_1 
			}
	}^N
	\frac{1}{N}
	\E\big(
	u^{(\varepsilon_1)}_{\sigma_1\circ\varepsilon_1(i_1,j_1)}\cdots u^{(\varepsilon_{2n})}_{\sigma_{2n}\circ\varepsilon_{2n}(i_{2n},j_{2n})}\big).\nonumber
\end{align}

Denoting $(k_m,l_m)=\sigma_m\circ\varepsilon_m(i_m,j_m)$ for $m=1,2,\ldots,2n$ we rewrite \eqref{eqn:32} as 
\begin{align*}
\sum_{\substack{i_1,\ldots,i_{2n}\in\{1,\ldots,N\},\\
		j_1,\ldots,j_{2n}\in\{1,\ldots,N\},\\
		j_k=i_{k+1} \textrm{ for } k=1,\ldots,{2n}-1,\\
		j_{2n}=i_1 
	}
}^N\frac{1}{N}\E\left(u^{(\varepsilon_1)}_{k_1,l_1}\cdots u^{(\varepsilon_{2n})}_{k_{2n},l_{2n}}\right).
\end{align*}

With this notations as in Remark \ref{rem:24} we have
 \begin{align}\label{weigarten}
\E\circ \tr
\big(
U^{(\sigma_1,\varepsilon_1)}U^{(\sigma_2,\varepsilon_2)}
\cdots U^{(\sigma_{2n},\varepsilon_{2n})}
\big)
=
\sum_{p,q\in \mathcal{P}_2^\varepsilon(2n)}
\VN ,
\end{align} 
where
$ \overrightarrow{\sigma} = ( \sigma_1,  \sigma_2, \dots, \sigma_{2n}) $,
$ \varepsilon = (\varepsilon_1, \dots, \varepsilon_{2n}) $ 
and
\begin{align*}
\VN = \Wg_N (p, q)
\cdot\frac{1}{N}
| \AN|,
\end{align*}
and 
$ \AN $ 
is  the set 
\begin{align*}
\AN = \big\{
(i_1, j_1, \dots, i_{2n}, j_{2n}) \in [ N ]^{4n}: & 
\ i_1= j_{2n}, i_{s+1} = j_s
\textrm{ for each } s \in [ {2n} -1], \\
 &  \hspace { 1.4 cm } \textrm { and }
  k_s = k_{p(s)}, l_s = l_{ q ( s ) } , \textrm { for each } s \in [ {2n} ] \big\}
\end{align*}
and
$ \Wg_N (p, q) $ 
is the value of the unitary Weingarten function as reviewed in Remark \ref{rem:24}.

In the remaining part of this section we will study  the asymptotic behavior of the quantity 
$ \VN $.
The main result is summarized in the Proposition below. Next we spend most of this section on proving it.

\begin{proposition}\label{prop:31}
	For any $n>0$, and any sequences $(\sigma_s)_{s=1,2,\ldots,2n},(\varepsilon_s)_{s=1,2,\ldots,2n}$, we have that:
	
	$1^\mathrm{o}$ For any $p,q\in\mathcal{P}_2^{\varepsilon}(2n)$
	 we have
	 $ \VN = O( 1 )$ 
	as
	$ N \to  \infty $.
	
	$2^\mathrm{o}$ Suppose $ p, q \in \cP_2^{\epsilon}(2n) $. 
	Either 
	$\VN = O(N^{-1})$,
	 or for any block of 
	$p \vee q $
	of the form
	$\{ a_1, a_2, \dots, a_m \}$
	where
	$ a_1 < a_2 < \dots < a_m $, and where we identify
	$ a_{ m + 1} = a_1 $, 
	we have that 
	for each
	$ s \in [ m ] $,
	  \begin{align} \label{eq:321}
	\{ p(a_s), q(a_s) \} = \{ a_{  s - 1 }, a_{ s + 1 } \},
	\end{align}
	In particular when $(\ref{eq:321})$ holds we have
	$ \varepsilon_{ a_s} \neq 
	\varepsilon_{ a_{ s + 1 } } $
	for all
	$ s =1, 2, \dots, m $.
	
	$3^{\mathrm{o}}$ 	If 
	$p \vee q$
	is crossing then
	$\VN= O(N^{-1})$.

\end{proposition}

 	Before we proceed with proofs of all parts of the above Proposition we need to prove a technical lemma, for which we need additional notation.
 	
 	For $S\subseteq \{1,\ldots,2n\}$  define
\begin{align*}
\ANm{S}&=\{(i_s,j_s)_{s\in S}; \mathrm{\ there\ exists\ } (i_{r},j_{r})_{r\notin S} \mbox{ such that } (i_1,j_1,\ldots,i_{2n},j_{2n})\in\AN\},\\
\FNm{S}&=|\ANm{S}|.
\end{align*}

	That is for a fixed subset $S$ of $[2n]$ the number	$\FNm{S}$
	  counts the subsequences with indices picked according to the set $S$, that can be continued to a full sequence in 
	  $\AN$.
	   We will write
	    $\ANm{m}$ and $\FNm{m}$ 
	    for 
	    $\ANm{[m]}$ and $\FNm{[m]}$
	     respectively. Obviously we have
	      $\ANm{2n}=\AN$.

\begin{lemma}\label{fpq:1}
 With the notations above, we have that:
	\begin{enumerate}
		\item[(1)]  If 
		$ m \in S $,
		then
		\begin{align*}
		\FNm{ S \cup  \{ m + 1 \} } & \leq N \FNm{ S }, \\
		\FNm{ S \cup \{ m - 1 \}  } & \leq N \FNm{ S }.
		\end{align*}
		\item[(2)] If $ B $ is a block of 
		$ p \vee q  $,
		$ m  \in B $ 
		and 
		$ B \setminus \{ m \} \subseteq S $, then
		\begin{align*}
		\FNm{ S \cup \{ m \} } = \FNm{ S } .
		\end{align*}	
		\item[(3)] If 
		 $ m, m^\prime \geq 1 $
		  are such that
		    $ m + m^\prime \leq 2n $, 
		    then
		\begin{align*}
\FNm{ m  + m^\prime }\leq\FNm{ m  } N^{  m^\prime -|\{B\in p \vee q \,\mid\, m < \max( B ) \leq  m+ m^\prime \}|},
		\end{align*}
		where
		 $ \max(S) $
		  denotes the largest element of the set $ S $.
		  \item[(4)] If $S$ is an interval, that is for some $s,m>0$ it is of the form
		  $ S =\{s, s+1, \dots, s+m \} $  
		  and moreover we have
		  \begin{align}\label{lemma:3.2.4}
		  \FNm{S} = o\big( N^{  | S | + 1  -|\{B\in p \vee q \,\mid\, 
		  			B \subseteq S  \}|} \big) 
		  \end{align}
		  then
		  $ \FNm{2n} = o( N^{ 2n + 1 - |p \vee q | } ) $
		\end{enumerate}
\end{lemma}	      

\begin{proof}
Part (1):  Suppose 
 $ m \in S $, 
 then
 $ i_m, j_m \in \ANm{S} $
 so in order to extend the sequence to 
 $ \ANm{S \cup \{ m + 1 \} } $,
 at most we need to add
 $ i_{ m  + 1 } $ and $ j_{ m + 1 } $.
 Since 
 $i_{m+1}=j_m$
 by definition of $\AN $, we have at most $N$ choices for $ j_{m+1} $. 
 The argument for the second inequality is analogous.
 
Part (2). We have that 
 $ ( i_m, j_m ) = ( \sigma_m \circ \varepsilon_m )^{ -1} ( k_m, l_m ) $.
 Since
 $ m \in B $,
 there exist some
 $ a, b \in B \setminus \{ m \} $
 such that 
 $ m = p( a) = q( b) $,
 so
 $ k_m = k_a $ 
 and
 $ l_m =l_b $.
 Hence
 $ ( i_m, j_m) $
 is determined by 
 $ \big\{ ( k_s, l_s ): s \in B \setminus \{ m \} \big\} $, 
 and hence  is determined by
 $ \big\{ (i_s, j_s ): s \in S \big\}. $
 
 For part (3), it suffices to prove the result for
  $ m^\prime = 1 $
  and the general case will follow by iteration.
  Denote by 
  $ B_1 $
   the block of
    $ p \vee q  $
    that contains 
    $ m + 1 $.
 If 
 $ m + 1 $
  is the largest element in 
  $ B _1 $
  then 
   $ \{ B\in p \vee q \mid m < \max (B) \leq m+ 1 \} = \{ B_1 \} $
 and
  $ B_1 \setminus \{ m + 1 \} \subseteq [ m ] $. 
  So part (2) gives
   \begin{multline*}
   \FNm{m + 1} = \FNm{ [ m ] \cup \{ m + 1 \} } = \FNm{ m } \\
   = \FNm{ m } \cdot N^{ 1 - | \{ B \in  p \vee q \mid m < \max (B) \leq m+ 1 \}| } .
   \end{multline*}
  
  If
  $ m + 1 $
  is not the largest element in 
  $ B_1 $, 
  then
   $  \{ B \in p \vee q \mid  m < \max (B ) \leq m + 1 \}
    = \emptyset $
    and part (1) gives
     \begin{align*}
     \FNm{ m + 1} 
     = \FNm{ [ m ] \cup \{ m + 1 \} }
     \leq 
     \FNm{ m } \cdot N^{ 1 - | \{ B\in p \vee q \mid m < \max (B) \leq m+ 1 \}| } .
     \end{align*}

     Proof of Part (4), we will expand the interval by adding one element such that the new set is still an interval. Thus we define $S'=S\cup \{m\}$ or $S'=S\cup \{s+m+1\}$. We shall show that inequality \eqref{lemma:3.2.4} holds for $S'$, then iteration until we reach full interval $[2n]$ will show (4).
     
     Assume that $S'=S\cup \{m\}$, the case $S'=S\cup \{s+m+1\}$ is similar. Then we have two possible cases:
     \begin{itemize}
     	\item $|\{B\in p \vee q : \ 
     	B \subseteq S'  \}|=|\{B\in p \vee q : \ 
     	B \subseteq S  \}|+1$ and of course $|S'|=|S|+1$, so  there is a block containing $m$ which is a subset of $S'$. Then by (2) of this lemma we have $\FNm{S'}=\FNm{S}$ and \eqref{lemma:3.2.4} is satisfied. 
     	\item $|\{B\in p \vee q : \ 
     	B \subseteq S'  \}|=|\{B\in p \vee q : \ 
     	B \subseteq S  \}|$ and $|S'|=|S|+1$ and \eqref{lemma:3.2.4} follows form point (1) of the lemma.
     \end{itemize}  
\end{proof}

\begin{proof}[Proof of $1^\mathrm{o}$ of Proposition \ref{prop:31}]
	Since 
	$ \FNm{1} \leq N^2 $, the inequality from Lemma \ref{fpq:1} (3) also gives that
	for 
	$ 1 \leq m \leq n $,
	\begin{align}\label{eq:fnm}
	\FNm{m} \leq  N^{ m + 1-| \{ B \in p \vee q \,\mid\,  \max(B) \leq m   \}|}.
	\end{align}
	
	In particular,
	$\FN \leq N^{ n + 1 - | p \vee q|} $.
	On the other hand,
	$\VN=\tfrac{1}{N}\Wg_N(p,q)\FN$ 
	and
	$\Wg_N(p,q)=O(N^{-n+ | p \vee q |})$.
		
\end{proof}
Next we proceed with a proof of $2^\mathrm{o}$ of Proposition \ref{prop:31}. For convenience we state it as a lemma.
\begin{lemma}\label{lemma:c2}
 Suppose that 
$ p, q \in \cP_2^{\epsilon}(n) $ 
and 
$ p \vee q $
has the block
$( a_1, a_2, \dots, a_m )$
with
$ a_1 < a_2 < \dots < a_m $ with the identification 
 $ a_{ m + 1} = a_1 $.
Then, either $ \VN = O(N^{-1}) $, or relation $(\ref{eq:321}) $ 
holds true 
 for each
 $ s \in [ m ] $ 
i.e. 
  $
 \{ p(a_s), q(a_s) \} = \{ a_{  s - 1 }, a_{ s + 1 } \}.
 $

   In particular when $(\ref{eq:321})$ holds then
 $ \varepsilon_{ a_s} \neq \varepsilon_{ a_{ s + 1 } } $
  for all
   $ s =1, 2, \dots, m $.
\end{lemma}

\begin{proof}

 First, remark that (\ref{eq:321}) is equivalent to each 
 $ s \in [ m ] $ 
 satisfies the following property (again, we identify 
 $ a_{m+1} = a_1 $):
  \begin{align}\label{eq:c2:02}
   a_{s+1} \in \{ p(a_s), q(a_s)\}.
  \end{align}
  That (\ref{eq:321}) implies (\ref{eq:c2:02}) is obvious. For the converse, if 
  $ a_s = p (a_{s-1} ) $, 
  then, since $ p $ is a pairing,
   $ p(a_s) = a{s-1} $. 
   Then, since
    $ a_{s+1}\in \{ p(a_s), q (a_s)\} $,
    we get
    $ a_{s+1} = q(a_s) $,
    that is (\ref{eq:c2:02}). 
    The case 
     $ a_s = q (a_{s-1}) $ 
     is similar.

 To prove relation \eqref{eq:c2:02}, suppose that  is not satisfied by some 
 $ t \in [ m ] $.
  Via a circular permutation, we can suppose that
  $ 2n = a_{t + 1 } $. 
  Since (\ref{eq:c2:02}) does not hold true for  $ t $, 
  we get that both
  $ p(a_t) $ and $ q ( a_t ) $ 
  are strictly less than $ a_t $.
   Then (\ref{eq:fnm}) gives that
 \begin{align*}
  \FNm{a_t - 1 } \leq  N^{ a_t - | \{ B \in p \vee q \mid : \  \max(B) \leq a_t-1  \}|}.
  \end{align*}
 Since
 $ k_{a_t} = k_{p(a_t)} $ 
 and 
 $ l_{a_t } = l_{ q( a_ t ) }$
 and
 $ p(a_t ) , q ( a_t ) < a_t $,
 we have that 
 $ ( i_{a_t}, j_{ a_t } ) $  
 is uniquely determined by
 $ \{ ( i_s, j_s ): s < a_t \} $,
   i.e.
   $\FNm{ a_t }= \FNm{ a_t - 1 }. $
 
  Furthermore, Lemma \ref{fpq:1}(3) gives
  \begin{align*}
  \FNm{2n} &\leq \FNm{a_t} \cdot 
  N^{ 2n- a_t - | \{  B \in p \vee q \mid : \ a_t < \max(B) < 2n  \} | }
  \\
  &\leq N^{ a_t - | \{ B \in p \vee q \mid : \  \max(B) \leq a_t-1  \}|}
   \cdot 
  N^{ 2n- a_t - | \{  B \in p \vee q \mid : \   a_t < \max(B) \leq 2n  \}  | }
  \\
  & \leq
  N^{ 2n - |\{ B \in p \vee q \mid :\   \max(B) \neq a_t \} | }.
  \end{align*}
  But
   $ a_{t+1} = 2n > a_t $,
   so 
   $ a_t $
   is not the largest element in any block of 
   $ p \vee q $,
   therefore
   $ \FNm{2n} \leq N ^{ 2n - | p \vee q | } $, 
   hence
   $ \VN = O( N^{-1} ) $.

   \end{proof}

Finally we proceed with the proof of $3^\mathrm{o}$ from Proposition \ref{prop:31}.
\begin{lemma} \label{lemma:c1}
	With the notations as above,
	if 
	$ p \vee q $
	 is crossing then
	  $ \displaystyle \lim_{N\to\infty}\VN = 0 $.
\end{lemma}
\begin{proof}

Suppose that 
$ p \vee q $ 
has a crossing
$ (a, b, c, d )$,
i.e. 
$  a < b < c < d $
and there are two distinct block
$ B _1 $ 
and 
$ B _2 $
of 
$ p \vee q $
such that
$ a, c \in B_1 $ 
and 
$ b, d \in B _2 $.
It suffices to show that
\begin{align}\label{eq:c2:01}
\FNm{c}\leq 
 N^{c - | \{ B \in p \vee q \mid :  \  \max ( B ) \leq c \} | }
 = N^{ | [ c ] | - | \{ B \in p \vee q \mid : \ B \subseteq [ c ] \} | } 
\end{align}
and the conclusion follows from Lemma \ref{fpq:1}(4) with $S=[c]$.

Without loss of generality we can assume that 
 $ a $ and $ c $
 are consecutive elements in
  $ B_1 $.

  Applying  (\ref{eq:fnm}) we obtain that
   \begin{align*}
    \FNm{ b  - 1 } \leq  N^{ b - | \{ B \in p \vee q  \mid : \  \max(B) \leq b - 1  \}|}.
    \end{align*}
    
    Next we shall show that
    \begin{align}\label{eq:c2:03}
     \FNm{ [ b  - 1 ] \cup \{ c \}  } \leq  N^{ b + 1 - | \{ B \in p \vee q \mid : \  B \subseteq [ b - 1 ] \cup \{ c \}  \}|}.  
    \end{align}
    
    We shall prove  (\ref{eq:c2:03}) by considering two cases. First, suppose that
    $ c $ 
    is the largest element of the block
     $ B _1 $.
     Since 
     $ a $ 
      and 
      $ c $
      are consecutive elements of 
      $ B_1 $
       and 
       $ a < b $,
       it follows that
       $ B_1 \setminus \{ c \} \subseteq [ b -1 ] $. 
       and Lemma \ref{fpq:1}(2) gives
    \begin{align*} 
    \FNm{[ b -1 ] \cup \{ c \} } = \FNm{ [ b - 1 ] } 
    \leq 
    N^{ b - | \{ B \in p \vee q \mid: \  B \subseteq [ b -1 ] \} | }.
    \end{align*}
    But 
     $ B_1 $ 
     is the only block of 
     $ p \vee q $ 
     not contained in 
     $ [ b -1 ] $
     and contained in
      $ [ b -1] \cup \{ c \} $, 
      so  inequality (\ref{eq:c2:03}) follows.
      
      Next, suppose that 
      $ c $
      is not the largest element of 
      $ B _1 $. 
      Since 
      $ a $
       and
        $ c $
         are assumed to be consecutive elements in
          $ B_1 $,  according to Lemma \ref{lemma:c2} we have that
            $ c \in \{ p(a), q(a)\} $.
    
    If 
    $ c = p(a) $, 
    then
    $ k_a = k_c $.
    So the sequence
     $  (i_s, j_s )_ { s \leq b-1 } $
    uniquely determines $ k_c $.
    Therefore, for the each sequence 
     $  (i_s, j_s )_{s \leq b -1}  \in \ANm{b-1} $
     there are at most 
     $ N $ 
     couples 
     $ ( i_c, j_c ) =
     ( \sigma_c \circ \varepsilon_c )^{-1} (k_c, l_c ) $
     such that
     $  (i_s, j_s )_{ s \in [ b -1 ] \cup \{ c \} } 
     \in  \ANm{ [ b-1] \cup \{ c \} }$. 
     That is
      \begin{align*}
      \FNm{ [ b-1] \cup \{ c \} } \leq N \cdot 
      \FNm{ [ b -1 ]}
      \leq
       N^{ b+ 1 - | \{ b \in p \vee q \mid : \ B \in [ b -1 ] \} | }.
      \end{align*}
   But $ c $ is not the largest element of $ B _1 $, so 
   $ B_1 $
   is not contained in $ [ b -1 ] \cup \{ c \} $, 
   therefore inequality \eqref{eq:c2:03} follows from the relation above. The case 
   $ c = q( a ) $,
   i. e.
   $ l_a = l_c $
   follows from a similar argument. Thus we proved \eqref{eq:c2:03}.
   
   Now we are ready to prove \eqref{eq:c2:01}. Using (\ref{eq:c2:03}) and applying parts (1) and (2) of Lemma \ref{fpq:1} for 
   $ m = c - t  $
   and 
   $ S = [ b - 1] \cup \{ c - t+ 1 , c-t+2, \dots, c \} $, 
   where 
   $ t = 1, 2,  \dots, c - b  - 1 $,
   an induction argument on 
   $ t $
   gives that 
   \begin{align}\label{eq:c2:04}
   \FNm{ [c]\setminus \{ b \} }
    \leq 
 N^{c - | \{ B \in p \vee q \mid : \ B \in [c ] \setminus \{ b \} \} | }.
   \end{align}

 Since 
 $ i_b = j_{ b -1} $
 and 
 $ j_b = i_{b+1} $,
 we have that each tuple 
 $ (i_s, j_s )_{ s \in [ c ] \setminus \{ b \} } $
from 
$ \ANm{[ c ] \setminus \{ b \} } $
has a unique extension to a tuple from
 $ \ANm{[ c ]}$. Therefore
 $ \FNm{[c ]} = \FNm{ [ c ] \setminus \{ b \} } $
 we know that $b$ is in the same block with $d$ so it is not the last element of any block, hence \eqref{eq:c2:01},
 follows from (\ref{eq:c2:04}). 
 
\end{proof}


It is useful to note the asymptotics of $\VN $ for unpermuted matrix $ U_N $ , a special choice of $p$ and $q$ and alternating sequence $(\varepsilon_s)_{s=1,\ldots,2n}$.
\begin{remark}\label{remark:3:1}
	Fix $n>0$ and set
	$ \varepsilon_1 = 1 $
	and 
	$ \varepsilon_s \neq \varepsilon_{ s  + 1 }$
	for all 
	$ s \leq 2n -1 $.
	Consider the pair partitions 
	$ \widetilde{p}_{n} , \widetilde{q}_{n} \in \cP_2^{\varepsilon} ( 2n ) $
	defined by
	$  \widetilde{p}_{n}( 2n ) = 1 $
	and 
	$ \widetilde{p}_n( 2 t ) = 2 t + 1 $
	for 
	$1 \leq  t < n $,
	respectively
	$ \widetilde{q}_n ( 2 t  ) = 2t - 1 $ 
	for
	$ 1 \leq t \leq n $.
	Then
	$ \widetilde{p}_n  \vee \widetilde{q}_{n}  $ 
	has only one block, the entire set 
	$ [ 2n ] $,
	so
	\begin{align*}
	\Wg_N ( \widetilde{p}_n ,\widetilde{q}_n ) = N^{  - 2n + 1} 
	( - 1)^{ n - 1}
	C_{ n - 1 } + O (N^{ -n -1 }).     
	\end{align*}
	
	If, moreover,
	$ \sigma_{ s, N } = \textrm{Id} $
	for all $ s \in [ n ] $, 
	we have that
	\begin{align*}
	k_s = \left\{ 
	\begin{array}{ l l}
	i_s & \textrm{if } s \textrm{ is odd}\\
	j_s &  \textrm{if } s \textrm{ is even}
	\end{array}
	\right., \  \ 
	l_s = \left\{ 
	\begin{array}{ l l}
	j_s & \textrm{if } s \textrm{ is odd}\\
	i_s &  \textrm{if } s \textrm{ is even.}
	\end{array}
	\right.
	\end{align*}   
	Therefore the condition on $k$'s and $l$'s do not bring any new restrictions and we have
	\begin{align*}
	\mathcal{A}_{ \overrightarrow{\sigma}, \varepsilon, N}^{( \widetilde{p}_n , \widetilde{q}_n)} = \big\{
	(i_1, j_1, \dots, i_{2n}, j_{2n}) \in [ N ]^{4n}: & 
	\ i_1= j_{2n}, i_{s+1} = j_s
	\textrm{ for each } s \in [ 2n -1]  
	\big\},
	\end{align*}
	thus $|\mathcal{A}_{ \overrightarrow{\sigma}, \varepsilon, N}^{( \widetilde{p}_n , \widetilde{q}_n)}|=N^{2n}$. 
	This gives
	\begin{align}\label{eq:u:frcm:1}
	\mathcal{V}_{ \overrightarrow{\sigma}, \varepsilon, N}^{( \widetilde{p}_n, \widetilde{q}_n )} = 
	( - 1 )^{ n - 1 }
	C_{n-1} 
	+ O(N^{ -2 } ).
	\end{align}
	Observe that from \ref{rem:Rdiag} we have that alternating cumulants of the length $2n$ of Haar unitary element are equal to the limits of $\mathcal{V}_{ \overrightarrow{\sigma}, \varepsilon, N}^{(\widetilde{p}_n, \widetilde{q}_n )}$, we will use this fact later in the proofs. We also note that it is well known (see \cite{HiaiPetz}) that in the limiting $*$-distribution of $U_N$ the $*$-distribution of a Haar unitary element.
	$\hfill $ $ \square $
\end{remark}    	

 Another consequence of Proposition \ref{prop:31} is the uniform (in $\sigma_N $) asymptotic norm bound of $ U_N^{\sigma_N} $.   It is known that such results do not hold true for other classes, such as Wishart random matrices.
 
 \begin{corollary}\label{cor:36}
 There exist some 
 $ \sqrt{27} \geq M > 0 $ 
 such that for any sequence 
 $ ( \sigma_N )_N $ 
 with  each
 $ \sigma_N $
  a permutation on the set 
 $ [ N ] ^2 $
   we have that
 \begin{align}\label{eq:m:sigma}
 \limsup_{N \rightarrow \infty } \| U_N^{ \sigma_N } \| < M.
 \end{align}
 \end{corollary}

 \begin{remark}\label{rem:sharp_M}
 Recall that if $c$ is a circular operator then $c^*c$ has the Marchenko-Pastur distribution with parameter $1$, and thus has norm $4$; hence $\|c \| = 2$. Thus Theorem \ref{thm:12}  gives that the sharpest $M$ cannot be less than $2$.  
 \end{remark}
 
 \begin{proof}
 
  Denote by 
   $ NC ( 2n, alt )$ 
   the set of all non-crossing permutations on the set 
   $ [ 2n ] $ 
   such that if
    $ B = \{ a_1, a_2, \dots, a_m \} $
    is a block of such a partition with
    $ a_1 < a_2 < \dots <  a_m $, 
    then
    $ m $ is even and $ a_s $ has different parity from $ a_{s + 1} $ for each $ 1 \leq s \leq m-1 $.  As shown in Proposition \ref{prop:31},
    denoting 
     $ \varepsilon = (1, -1, \dots, 1, -1) $,
     we have that
 \begin{align*}
 \E \circ \tr \big( [ U_N^{\sigma_N} ( U_n^{\sigma_N})^\ast ]^n   \big)  &= 
 \sum_{\substack{  p, q \in P_2^{\epsilon}( 2n )\\ p \vee q \in NC ( 2n,  alt )} } \VN + O( N^{-1} )\\
  & = \sum_{ \pi \in NC (2n, alt) }  \big(  
 \sum_{ \substack{ p, q \in P_2^\epsilon ( 2n ) \\ p \vee q = \pi }} 
 \VN  \big) + O( N^{-1} ).
 \end{align*}

 Let
  $ \pi \in NC ( 2n , alt ) $ 
 and 
 $ p, q \in P_2^\epsilon ( 2n ) $
 be such that 
 $ p \vee q  = \pi $. 
 Since 
  $\VN=\tfrac{1}{N}\Wg_N(p,q)\FN$ 
 and, as seen in the proof of Proposition 3.1, 
 $\FN \leq N^{ n + 1 - | p \vee q|} $, 
 we have that
  $ | \VN | $ 
 is bounded by the absolute value of the leading term in the expansion of 
 $ \Wg_N (p, q) $, 
 that is 
  \begin{align*}
  | \VN | \leq \prod_{ B \textrm{ block in } \pi } Cat_{\frac {| B |}{2} -1 }.
 \end{align*}
 On the other hand, 
 $ \VN  = O(N^{-1}) $
 unless the restrictions  of
  $ p $ and $ q $
  to each block of
  $ \pi $
  satisfy the relation (\ref{eq:321}) from part $2^\mathrm{o}$ of Proposition \ref{prop:31}. 
  Furthermore, note that if 
  $ B $
   is a block of 
  $ \pi $
  with exactly 2 elements, then the restrictions of $ p $ and $ q $ to $ B $ 
    are uniquely determined by
   relation (\ref{eq:321});  if 
  $ B $ is a block of 
  $ \pi $
   of length greater than $ 2 $, then relation (\ref{eq:321}) allows for at most two possible restrictions to $ B $ of the pair
    $(p. q) $,
     one when 
  $ a_{2}= p(a_1) $ 
  and one when 
  $ a_2 = q(a_1) $.
 Henceforth, using that 
 $ \displaystyle  Cat_p < \frac{4^{2p}}{2} $,
 we obtain 
 \begin{align*}
  \limsup_{ N \rightarrow \infty } | \E \circ \tr \big( [ U_N^{\sigma_N} ( U_n^{\sigma_N})^\ast ]^n   \big) |   \leq  & 
  \sum_{ \pi \in NC(2n, alt) } 2^{ | \{ B \textrm{ block in } \pi : \  | B | > 2 \} | } \cdot
  \prod_{ B \textrm{ block in } \pi } Cat_{ \frac{| B |}{2} -1 } \\
   & \  \hspace{1cm} < 
  | NC(2n, alt) | \cdot 4^n.
 \end{align*}
 
 Next, note that 
 $ NC( 2n, alt ) $ 
 coincides with the set of non-crossing partitions on 
 $ [ 2n ] $
  with all blocks of even length, hence
  $ \displaystyle | NC( 2n, alt) | 
  =  \frac{1}{2n +1} \binom{3n}{n} $
  (see, for example sequence  A001764 in \cite{OEIS})).  
Furthermore, since  
    $\displaystyle \binom{3n}{n} < \frac{3^{3n}}{4^n}$, 
    it follows that 
 \begin{align*}
   \limsup_{ N \rightarrow \infty } | \E \circ \tr \big( [ U_N^{\sigma_N} ( U_n^{\sigma_N})^\ast ]^n   \big) | 
   \leq 27^n,
   \end{align*}
   hence the conclusion.
 \end{proof}


\section{Asymptotic distribution of permuted Haar unitary matrices}
\label{sec:asymptotic_distribution}
In this section we prove Theorem \ref{thm:11} from the introduction. We prove separately the claim about the $\ast$--distribution of a permuted matrix and the claim about asymptotic freeness. The two theorems of this section combined give the result announced in the introduction. We identify conditions on a deterministic sequence $(\sigma_N)_N$ of permutations of matrices, under which the sequence $\left(U_N^{\sigma_N}\right)_N$ of permuted Haar unitary matrices is with $N\to\infty$: (1) asymptotically circularly distributed, (2) asymptotically free from the unpermuted sequence $(U_N)_N$. We state the conditions below. In Sections 6 we shall show that the the results of this section hold also in the almost sure sense. Moreover in Section 7 we show that the conditions for asymptotic circular distribution and asymptotic freeness are satisfied almost surely by a sequence of random permutations, showing that these are rather generic and general conditions.

\begin{notation}\label{not:41}
For
$ N \geq 1 $,
let 
 $\sigma$
	  and
	  $ \mu $
	   be  permutations on 
	   $[N]\times[N]$.

  $1^\mathrm{o}$ Define
 \begin{align*} 
 	X_N(\sigma,\tau)=|\{(i,j,k) \mbox{ such that } \sigma(i,j)\in\{\tau(i,k),\tau(k,j)\}|.
 	\end{align*}
 	In particular $X_N(\sigma,id_N)$ is the number of elements which do not change row or column after the permutation $\sigma$.
 	
 	$2^\mathrm{o}$ Define
 		\begin{align*}
 		Z_N(\sigma)
 		&=| \big\{ (i,j, k, l ) \in [ N ]^4 : \
 			\sigma (i, j) =
 			\big( \pi_1 \circ \sigma (i, l), 
 			\pi_2 \circ \sigma ( k, j) \big)
 		 \big\} |\\
 		&+| \big\{ (i,j, k, l ) \in [ N ]^4 : \
 			\sigma (i, j) = 
 			\big( \pi_1 \circ \sigma ( k, j),
 			\pi_2 \circ \sigma (i, l) 
 			\big)
 			  \big\} |.	
 		\end{align*}
 
	$3^\mathrm{o}$ 
	   Denote
	   \begin{enumerate}
	   \item[${}$]$ Y_{1,N}(\sigma, \mu) =
	   | \big\{  
	   (i, j, k) \in [ N ]^3:\
	    \pi_1 \circ \sigma (i, j) \in 
	    \{ 
	    \pi_1 \circ \mu (i, k), 
	    \pi_1 \circ \mu ( k, j)   \} 
	   \big\} |, $
	   \item[${}$]
	   $ Y_{2,N}(\sigma, \mu) =
	   	   | \big\{  
	   	   (i, j, k) \in [ N ]^3:\
	   	    \pi_2 \circ \sigma (i, j) \in 
	   	    \{ 
	   	    \pi_2 \circ \mu (i, k), 
	   	    \pi_2 \circ \mu ( k, j)   \} 
	   	   \big\} |, $
	   \end{enumerate}
	   and
	   \[ 
	   Y_N( \sigma, \mu ) = Y_{1,N}(\sigma, \mu ) 
	   + Y_{2,N} ( \sigma , \mu ).
	   \]
	Thus $Y_N(\sigma,\mu)$ counts the number of elements that were in the same row or column and permutations $\sigma,\mu$ take them to the same row or column.
\end{notation}
 
 \begin{definition}\label{def:permprop}
 	Let
 	 $\left(\sigma_N\right)_{N\in\mathbb{N}}$ 
 	 be a sequence of permutations, such that
 	  $\sigma_{N}\in \mathcal{S}( [ N ]^2 )$
 	   for 
 	   $N\geq 1$.
 	We say that the sequence $(\sigma_N)_N$ satisfies property:
 	\begin{enumerate}
 	  \item[(C1)] if \quad 
 	  	$\ds
 	   		\lim_{N \rightarrow \infty } \frac{ Z_N(\sigma_N) }{ N^4 } = 0,$
 	   		 \bigskip
 	   		 		
 	  	\item[(C2)] if \quad 
 	    $\ds
 	   	\lim_{N \rightarrow \infty } \frac{X_N(\sigma_N,id_N)}{N^2} = 0,	$
 	   	\bigskip
 	   	
 \item[(C)  ] if \quad
  		$\ds
  		\lim_{N \rightarrow \infty } \frac{ Y_N(\sigma_N , \sigma_N ) }{ N^3 } = 0.$

 	\end{enumerate}
 \end{definition}
 
 \begin{remark}\label{c:c1c2}
Condition
  (C) 
 imply both conditions 
  (C1) 
  and 
   (C2) .
  
  First, it is straightforward to observe that for any permutation
   $\sigma$  of $[N]^2$
    we have
     $Z_N(\sigma)\leq N Y_N(\sigma,\sigma)$. Indeed any tuple 
     $(i,j,k)$ 
     which was counted by 
     $Y_N(\sigma,\sigma)$
      will contribute to
       $Z_N(\sigma)$
        and there are at most 
        $N$ 
        ways to choose the fourth element in 
        $(i,j,k,l)$.
        
    Next, if for each
     $ i \in [ N ] $
     we denote
     \begin{align*}
     a_i  & = | \big\{ j \in [ N ] : \ 
      \pi_1\circ \sigma ( i, j) = i   \big\} |
      \\
      b_i  &
       = | \big\{ 
        j \in [ N ]: \ 
         \pi_2 \circ \sigma ( j, i) = i 
       \big\} |,
     \end{align*}
     we have that
      $ \ds 
      Y_{1, N} ( \sigma, \sigma ) 
      \geq \sum_{ i =1}^N a_i^2 $    
      and
     $ \ds 
     Y_{ 2, N } ( \sigma, \sigma) 
     \geq \sum_{i,j=1}^N 
     b_i^2 $, 
     hence
     \begin{align*}
   X_N ( \sigma, id_N ) =
   \sum_{ i=1}^N ( a_i + b_i)
   \leq 
   \frac{1}{N} \sum_{i,j=1}^N a_i^2 + b_i^2
   \leq \frac{1}{N} Y_N ( \sigma).
     \end{align*}
     
    In Section 7 we will show that conditions from Definitions  \ref{def:permprop} and \ref{def:pairperm} are almost surely satisfied by a sequence of random permutations. This remark implies that it is enough to check that 
    (C)
     is generic and it follows that so are
     (C1)
     and 
     (C2). 
     This observation justifies also reference to condition
      (C)
       in the Introduction.  
\end{remark} 

\begin{remark}
	Condition 
	(C2)
	 implies that 
	$ Y( \sigma_{s, N},id_N) = o(N^3).$
	
	Indeed, if
	$ \sigma_{s, N} $ 
	satisfies 
	 (C2),
	  then
	\begin{align*}
	Y_1( \sigma_{s,N}, \textrm{Id}) & = 
	| \big\{
	(i, j, k) \in [ N ]^3:\ 
	\pi_1 \circ \sigma_{s, N} (i, j) \in \{ i, k \} 
	\big\}|\\
	\leq & |\big\{
	(i, j) \in [ N]^2:\ \pi_1 \circ \sigma_{s, N}(i, j) = i 
	\big\}|
	+ 
	| \big\{
	(i, j, k) \in [ N ]^3:\ k = \pi_1 \circ \sigma_{s, N}(i, j)
	\big\}|
	\\
	=  & o(N^2) + N^2,
	\end{align*}
	and similar computations hold  true also for 
	$ Y_2( \sigma_{s, N}, \textrm{Id})$.
	
\end{remark}

\begin{definition}\label{def:pairperm}
	Let
	$\left(\sigma_N\right)_{N\in\mathbb{N}},\left(\tau_N\right)_{N\in\mathbb{N}}$  
	be two sequences of permutations, such that
	$\sigma_{N},\tau_N \ab \in \mathcal{S}( [ N ]^2 )$
	for 
	$N\geq 1$.
	We say that the pair of sequences 
	$(\sigma_N)_N, ( \tau_N )_N $
	 satisfy property  
\begin{align} \mathit{if\ } \tag{C3}
		\lim_{N \rightarrow \infty } \frac{ Y_N(\sigma_N , \tau_N ) }{ N^3 } = 0
		\mathrm{\ and \ }
		\lim_{N \rightarrow \infty } \frac{X_N(\sigma_N,\tau_N)}{N^2} = 0.		\end{align}
\end{definition}

The next theorem proves the first claim of Theorem \ref{thm:11}. Observe that it is enough to consider any monomial in $U_N^{\sigma_N}$ and $\left(U_N^{\sigma_N}\right)^*$ instead of a general polynomial.
\begin{theorem}\label{thm:circ}
If 
$(\sigma_N)_N$
 satisfies
  \emph{(C1)},
   then
 $ U_N^{\sigma_N} $ 
 is asymptotically circularly distributed with variance 1.
\end{theorem}

\begin{proof}
 Fix 
 $ n \in \mathbb{N} $ 
 and $ \varepsilon : [ 2n ] \rightarrow \{ 1, \ast \} $
  and denote
  \begin{align*}
  NC_2^\varepsilon ( 2n ) = &
  \big\{  \pi \in NC_2(2n) 
  \textrm { such that }
   \varepsilon_{ \pi (s) } \neq \varepsilon_s  
   \textrm { for all } s \in [2n ] 
   \big\}.
  \end{align*}
  
 It suffices to show that 
 \begin{align}\label{thm:01:01}
 \lim_{ N \rightarrow \infty }
\E\circ \tr
\left( ( U_N^{\sigma_N})^{\varepsilon(1)} 
( U_N^{\sigma_N})^{\varepsilon(2)} 
\cdots 
( U_N^{\sigma_N})^{\varepsilon(2n)} 
\right)
=  | NC_2^\varepsilon ( 2n ) |,
 \end{align}
 and the conclusion follows because these are the $*$-moments of circular operator of variance 1.
 
 According to Lemma \ref{lemma:c1}, we have that
 \begin{align} \label{thm:01:02}
  \lim_{ N \rightarrow \infty }
 \E\circ \tr
 \left( ( U_N^{\sigma_N})^{\varepsilon(1)} 
 ( U_N^{\sigma_N})^{\varepsilon(2)} 
 \cdots 
 ( U_N^{\sigma_N})^{\varepsilon(2n)} 
 \right)
 = 
 \sum_{\substack{p,q\in \mathcal{P}_2^\varepsilon(2n)\\ p \vee q \textrm { is non-crossing}}}
  \lim_{N \rightarrow \infty}
  \VN,
 \end{align}
 so it suffices to show the equality  of the  right-hand terms from (\ref{thm:01:01}) and (\ref{thm:01:02}).
 
 From the construction of the partition 
 $ p \vee q $,
  we have that 
 $ (a, b) $ is a block of length $ 2 $ in
  $ p \vee q $
  if and only if 
  $ p(a) = q (a) = b $. 
  So 
  $ p \vee q \in NC_2^\varepsilon ( 2n ) $
  if and only if 
  $ p = q $. 
  Conversely, given 
  $ \pi \in NC_2^\varepsilon(2n) $,
  taking 
  $ p = q = \pi $
  gives that
  $ p \vee q = p = \pi $.
  Hence
  it suffices to show that condition 
  (C1)
   implies that
   $ \displaystyle 
    \lim_{N \rightarrow \infty } \VN = 0 $
    unless 
  $ p = q $.
  
  We shall prove the statement above by induction on $ n $. For
   $ n =1 $ 
   the property is trivial.
    Fix 
     $ n \geq 2 $ 
     and suppose that 
     $ p, q $  are two pairings from 
     $ \cP_2^\varepsilon ( 2n ) $ 
     such that 
      $\displaystyle 
      \limsup_{N \rightarrow \infty} |\VN | > 0$.
      From Lemma \ref{lemma:c1}, the partition
      $ p \vee q $ 
      is then non-crossing, hence (see, for example \cite[Remark 9.2.2]{NicaSpeicherLect}) it has a block
      of the form
      $ ( a, a+1, a+2, \dots, a+ m ) $. 
      Then, via a cyclic permutation on the set
      $ [ 2n ] $, 
      we can suppose that
      $ (1, 2, \dots, m ) $ 
      is a block in 
      $ p \vee q $.
      
      We shall show that condition
      (C1)
       implies that if $p,q$ are such that $p \vee q$
        has a block with more than two elements then 
        $\VN\to 0$ 
        as 
         $N\to\infty$.
      Suppose first that
       $ m > 2 $. 
       Since 
       $ m $ 
       is even, we have that 
       $ m \geq 4 $;
       moreover, Lemma \ref{lemma:c2} gives that
       $ \varepsilon(s) \neq \varepsilon (s+1) $
       for all $ s \in [ m -1 ] $.
       Therefore there exists 
        $ s \in \{1,\ldots,m-1\} $
        such that
        $ \varepsilon(s-1) = \varepsilon (s+1) =1 $
        and
        $ \varepsilon(s) =\ast $, 
        that is
         $( k_{s-1}, l_{s-1}) = \sigma_N ( i_{s-1}, j_{s-1}),\,$
         $ (k_s, l_s ) = \sigma_N ( j_s, i_s )$
         and
         $ ( k_{s+1}, l_{s+1} ) = \sigma_N ( i_{s+1}, j_{s+1}) $.

    Denote 
    $ i: = j_s = i_{s+1} $,
    $ j := j_{s-1} = i_s $,
     $ k := i_{s-1} $ 
     and
     $ l := j_{s+1} $.
      Lemma  \ref{lemma:c2} gives also that
          $ \{p(s),  q (s ) \} = \{ s-1, s+1 \} $.
     Consider two cases:
     \begin{itemize}
     	\item If
     	$ ( p(s), q(s)) = (s-1, s+1 ) $, 
     	then
     	$ k_s = k_{s-1} $ 
     	and
     	$ l_s = l_{s+1} $,
     	so
     	\begin{align*}
     	\FNm {\{ s, s-1, s+1 \} }
     	\leq | \big\{ (i, j, k, l):\ 
     	\sigma_N (i, j) = ( k_s, l_s ) = 
     	\big( \pi_1 \circ \sigma_N ( k, j),
     	\pi_2 \circ \sigma_N ( i, l) \big) 
     	\big\}|
     	\end{align*}   
     	and condition
     	 (C1)
     	  gives that
     	$ \FNm { \{ s, s-1, s+1 \} } = o(N^4)$
     	hence Lemma \ref{fpq:1}(4) implies that
     	$ \displaystyle \lim_{N \rightarrow \infty } 
     	\VN = 0 .$
     	\item     Similarly, if
     	$(p(s), q(s) ) = (s+1, s-1) $, 
     	then 
     	$ k_s = k_{s+1} $
     	and
     	$ l_s = l_{s-1}$, 
     	therefore
     	\begin{align*}
     	\FNm { \{ s, s-1, s+1 \} }
     	\leq | \big\{ (i, j, k, l):\ 
     	\sigma_N (i, j) = ( k_s, l_s ) = 
     	\big( \pi_1 \circ \sigma_N ( i, l),
     	\pi_2 \circ \sigma_N ( k, j) \big) 
     	\big\}|
     	\end{align*}   
     	and again condition
     	(C1) 
     	 and   Lemma \ref{fpq:1}(4) imply that
    	$ \displaystyle \lim_{ N \rightarrow \infty } 	\VN  = 0.$
     	
     \end{itemize}

    The only case remaining is
    $ m = 2 $. 
    then $ (k_1, l_1) = (k_2, l_2 ) $
    and, since 
    $ \varepsilon_1 =  \varepsilon_2 \circ \ast  $,
     we have that 
     \[
     (i_1, i_2) = \big( \sigma_N \circ \varepsilon_1 \big) ^{ -1} 
     (k_1, l_1) = 
  \ast \circ \varepsilon_{ 2 }^{-1} \circ \sigma_N^{-1} ( k_1, l_1)
      =  \ast ( i_2, i_3) 
     \]
      that is
       $ i_1 = i_3 $. Finally we apply the induction hypothesis to the partitions obtained by removing the block 
       $ (1, 2) $
       from 
       $ p \vee q $ and the conclusion follows.
\end{proof}

  \begin{remark}\label{remark:4:1}
 The proof of Theorem \ref{thm:circ} above does not use that 
  $ \sigma_{ s, N } = \sigma_N $ 
  for 
  $ s > m $. 
  In fact, we have proved a more general technical result from below, which we shall use later in this section.

  \emph{ Suppose that}
   $ \big( \sigma_N \big)_N $ 
   \emph{is a sequence of permutations that satisfies condition
    \emph{(C1)}
   and that}
   $ \sigma_{ s, N } = \sigma_N $
   \emph{for all}
   $ s \in [ m ] $.
   \emph{Suppose also that}
   $ p, q \in \cP_2^\varepsilon (2n) $ 
   \emph{are such that}
   $ [ m ] = \{ 1, 2, \dots, m \} $
   \emph{ is a block of }
   $ p \vee q  $,
   \emph{i.e.}
   $ p \vee q $
   \emph{is the juxtaposition}
   $ [ m ] \oplus \tau $
   \emph{of the block}
   $ [ m ] $
   \emph{and of}
   $ \tau $,
   \emph{some partition on}
   $ [ n ] \setminus [ m ] $.
   \emph{Then:}
   \begin{enumerate}
   \item[(i)] \emph{If} 
    $ m > 2 $
    \emph{then}
    $ \displaystyle 
   \lim_{N \rightarrow\infty }  \VN = 0 $.
   \item[(ii)]\emph{If}
   $ m =2 $, 
   \emph{then} 
   $ p(1) = q (1) = 2 $, \textit{both} 
   $\ds\lim_{N \rightarrow \infty} \VN$ \textit{and}  $\ds\lim_{ N \rightarrow \infty }
   \mathcal{V}_{ \overrightarrow{\sigma}^\prime, \varepsilon^\prime} ( \widetilde{ p }, \widetilde{ q } )$ \textit{exist and are equal}
   \emph{where 
   $ \overrightarrow{\sigma}^\prime $, 
  $ \varepsilon^\prime $,
   $ \widetilde{ p } $
   respectively
   $ \widetilde{ q } $
   are the restrictions of
    $ \overrightarrow{\sigma} $,
    $ \varepsilon $,
   $ p $, 
   respectively 
   $ q $
   to the set}
  $  \{ 3, 4, \dots, n \}.  $  $\hfill \square$
   \end{enumerate}
  \end{remark}

The next Theorem gives a sufficient condition on the sequence
 $ \big( \sigma_N \big)_N $ 
 such that
  $ U _N $ 
  and 
  $ U_N^{ \sigma_N } $
  are asymptotically free. 
  For the proof, we shall need the following technical result.
  
  \begin{lemma}\label{tech:main:01}
  	Let
  	$ p, q \in \cP_2^\varepsilon (2n ) $
  	be such that
  	$ [ m ] = \{ 1, 2, \dots, m \} $
  	is a block of
  	$ p \vee q $.
  	Suppose also that
  	 $ \displaystyle 
  	 \limsup_{ N \rightarrow \infty}
  	 \VN > 0 $
  	 and that there exists 
  	 $ s\in\{2,\ldots, m \}$ 
  	 such that 
  	 $ \sigma_{s-1} = \sigma_{ s} = \textrm{Id} $.
  	 We have then:
  	 \begin{enumerate}
  	 	\item[(i)] if 
  	 	$ \varepsilon_s = 1 $,
  	 	then
  	 	 $ s = p( s -1 ) $ 
  	 	 and 
  	 	  $ i_s = k_{ s - 1 } = k_{ s } $;
  	 	  \item[(ii)] if
  	 	   $ \varepsilon_s = \ast $, 
  	 	    then
  	 	    $ s = q ( s -1 ) $
  	and
  	 $ i_s = l_{ s -1} = l_s $. 	    
  	 \end{enumerate}
  	  \end{lemma}
  	  
  \begin{proof}
  	 If 
  	  $ m =2 $,
  	   the result is trivial.
  	   
  	   Suppose that 
  	    $ m > 2 $ and
  	    $ \varepsilon_s = 1 $.
  	    Lemma \ref{lemma:c2} gives that
  $ \varepsilon_{ s -1 }= \ast $.
  So
  	$ l_s = i_{ s + 1 } $,
   $ l_{s-1} = i_{ s -  1 }  $
   and
   	$ k_s =k_{ s - 1 } =  i_s $.  
   	
   	If $ s \neq p ( s - 1 ) $,
   	 applying again Lemma \ref{lemma:c2}, we obtain 
   	 $ s = q ( s -1 ) $, 
   	 i.e. $l_s=l_{s-1}$ which in turn implies that
   	 $ i_{s-1} = i_{ s + 1 } $,
   	 so
   	 $\FNm { \{ s-1, s \} }  \leq N^2 $
   	 and the result follows from Lemma \ref{fpq:1} (4).

   	 The case 
   	 $ \varepsilon_s = \ast $
   	 is similar.   
  \end{proof}

The next theorem proves the second claim of Theorem \ref{thm:11} and Theorem \ref{thm:14}, again it is enough to consider monomials as by linearity we get the claim for any polynomial.

\begin{theorem}\label{thm:main:2}
 Let
  $ \big\{ 
  \big( \sigma_{t, N} \big)_{ N>0} : 1 \leq t \leq r
   \big\} $
  be a family of 
  $ r $ sequences of permutations 
  such that each
  $ \big( \sigma_{t, N} \big)_{ N>0} $
  satisfies conditions 
  \emph{(C1)}
  and
   \emph{(C2)}
   then
  $ (U_N)_N $ 
  is asymptotically free from the family 
  $ \big\{ \big( U_N^{ \sigma_{s, N}} \big)_N :
  1 \leq s \leq r 
  \big\} $.
  
  Moreover if each pair $(\sigma_{s,N}),(\sigma_{t,N})$ for
  $1\leq s \neq t \leq r$, satisfies
   \emph{(C3)}
     then
   $ \big\{ \big( U_N^{ \sigma_{s, N}} \big)_N :
    1 \leq s \leq r 
   \big\} $
   forms a family of asymptotically free random matrices.  
\end{theorem}
                                                       Before proceeding with the proof, we will make the following remark.
\begin{remark}
Conditions
(C1)
and
(C2)
are sufficient, but not necessary for the 
asymptotic free independence of 
$ U_N $ 
and
$ U_N^{ \sigma_N } $.
Example \ref{d-part-transp} from next section describes a sequence 
$ ( \sigma_N)_N $ 
which does not verify 
(C1)
but asymptotic free independence still holds. 
An example of sequence of permutations not verifying 
(C2)
 but asymptotic free independence still holding true is given by each  
$\sigma_N$
 sifting circularly the $i$-th row of a matrix by $i-1$ to the right, i.e. 
 $\sigma_N(i,j)=\big(i,j+i-1 (\mod N) \big) $.
 
\end{remark}

\begin{proof}[Proof of Theorem \ref{thm:main:2}]

 Let $ n $ be a positive integer, 
 $ \varepsilon : [ n ] \rightarrow \{ 1, \ast \} $
 and
 $ \theta : [ n ] \rightarrow \{ id_N , 
 \sigma_{1, N}, \dots, \sigma_{r, N} \} $.
  From (\ref{weigarten}) on page \pageref{weigarten}, we have that
  \begin{align*}
  \E \circ \tr \big(
  U^{ ( \theta_1, \varepsilon_1 ) }
  U^{ ( \theta_2, \varepsilon_2 ) }
  \cdots
  U^{( \theta_n, \varepsilon_n )}
   \big) 
   = \sum_ {p, q \in \cP_2^\varepsilon ( n ) } 
   \VN.
  \end{align*}
  
  If 
  $ p \vee q $
  is crossing, then by Lemma \ref{lemma:c1}, we have that
  $ \displaystyle 
   \lim_{N \rightarrow \infty } \VN = 0 $.
   So the equation above becomes
   \begin{align}
  \lim_{N \rightarrow \infty }
   \E \circ \tr \big(
    U^{ ( \theta_1, \varepsilon_1 ) }
    U^{ ( \theta_2, \varepsilon_2 ) }
    \cdots
    U^{( \theta_n, \varepsilon_n )}
     \big) 
   =
   \sum_{\substack{p, q \in P^\varepsilon_2 (n )\\ p \vee q \textrm{ non-crossing }}}
  \lim_{ N \rightarrow \infty } \VN.
   \end{align}

 Consider
  $( \mathfrak{A}, \varphi ) $
  a non-commutative probability space containing
   the $*$-free 
   $ r + 1 $-tuple 
   $ \{ x_0, x_1, \dots, x_r \} $
   with
    $ x_0 $ 
    a Haar unitary and
    $ x_1, \dots, x_n $
    a circular elements of variance $ 1 $. 
    For 
    $\tau 
    \in \{ id_N, \sigma_{1,N}, \dots, 
    \sigma_{n,N}\}$
     and 
     $\eta \in \{*, 1\}$,
     denote
    \begin{align*}
    x^{( \tau, \eta)} = \left\{
    \begin{array}{cl}
    x_0  & \textrm{ if } 
    ( \tau, \eta ) = ( id_N , 1 ) \\
    x_0^\ast  & \textrm{ if } 
        ( \tau, \eta ) = ( id_N , \ast ) \\
    x_s  &  \textrm{ if } 
        ( \tau, \eta ) = ( \sigma_{s, N}, 1 ) \\ 
    x_s^\ast  & \textrm{ if } 
        ( \tau, \eta ) = ( \sigma_{s, N}, \ast ) .
    \end{array}
    \right.
    \end{align*}
    
    With the assumption that the range of $ \theta $ contains
     $id_N $
      only if each 
      $ \big( \sigma_{s, N})_N $
      satisfies condition ($C1$), 
       it suffices to show that for each
    $ \omega $
     non-crossing partition on 
     $ [ n ] $,
    \begin{align}
    \lim_{ N \rightarrow \infty } 
    \sum_{ \substack{ p, q \in \cP_2^\varepsilon ( n ) \\ p \vee q = \omega }}
    \VN = 
    \kappa_{ \omega } \big[ x^{ ( \theta_1, \varepsilon_1 )},
    x^{ ( \theta_2, \varepsilon_2 )},
    \cdots,
    x^{ ( \theta_n, \varepsilon_n )}
      \big].
    \end{align}

   Fix $\omega \in NC(n)$ as above. Then 
  $ \omega $
  has a block with consecutive elements (see \cite[Remark 9.2.2]{NicaSpeicherLect}) and 
   via a cyclic permutation, we may suppose that this block is
    $ [ m ] = (1, 2, \dots, m ) $.
    Then 
    $ \omega $ 
    can be written as the juxtaposition
     $ \omega_1 \cup\omega_2 $
      where $\omega_1$ consists of the single block
    $ [ m ] $
    and $\omega_2$ is
     a non-crossing partition
      of the set
      $[ n ] \setminus [ m ]= \{m +1 , \dots, n\}$.
      Denoting by
      $ \varepsilon^\prime $, 
      respectively 
      $ \overrightarrow{\sigma}^\prime $
      the restrictions of
       $ \varepsilon $,
       respectively
       $ \overrightarrow{\sigma}$
       to the set 
       $ [ n ] \setminus [ m ] $,
    it suffices to show that
       \begin{align}\label{eq:main2:02}
     \lim_{ N \rightarrow \infty } 
    \sum_{ \substack{ p, q \in \cP_2^\varepsilon ( n ) \\ p \vee q = \omega }}
    \VN & \\
     =
    \kappa_m 
    \big(   x^{( \theta_1, \varepsilon_1)}, &
     x^{( \theta_2, \varepsilon_2)},
     \dots,
     x^{( \theta_m, \varepsilon_m)}
    \big)
    \cdot
      \lim_{ N \rightarrow \infty } 
     \sum_{ \substack{ 
    p^\prime, q^\prime \in \cP_2^{\varepsilon^\prime} ( n-m ) \\ p^\prime \vee q^\prime = \omega_2 }}
     \mathcal{V}_{ \overrightarrow{\sigma}^\prime, \varepsilon^\prime, N } (p^\prime, q^\prime ) \nonumber
       \end{align}
    and the result will follow by induction on 
     the number of blocks of $\omega$. 
     
     If 
     $ m $ 
     is odd, then the left hand side of 
     (\ref{eq:main2:02}) trivially vanishes as there are no
     $ p, q \in \cP_2^\varepsilon ( n ) $
     such that
     $ p \vee q $
     has blocks of even length. Moreover, from Lemma \ref{lemma:c2},
     the left hand side also vanishes if
      $ \varepsilon $
       is not  alternating on 
       $ [ m ] $, 
       i.e.
        $ \varepsilon_s \neq \varepsilon_{ s + 1} $
        for all
        $ s \in [ m -1 ]$. Let us show that is either of these conditions holds then the right hand side of (\ref{eq:main2:02}) also vanishes. 
        If 
        $ \theta $ 
         is not constant on 
         $ [ m ] $, the right hand side vanishes from the assumed $*$-free independence of the set 
         $\{ x_0, x_1, \dots, x_r \} $,
     and if 
         $ \theta $ 
         is constant on 
          $ [ m ] $ but $ \varepsilon $
       is not  alternating on 
       $ [ m ] $ then the right hand side vanishes 
          from the fact that circular operators and Haar unitaries  are
          $ R $-diagonal (i.e. only their alternating free cumulants of even order may be different from zero, see for example \cite[Lect. 15]{NicaSpeicherLect}). It suffices then to prove the equality in
          (\ref{eq:main2:02})
          under the additional assumptions that 
          $ m $
          is even and 
          $ \varepsilon_s \neq \varepsilon_{ s + 1 } $
          for all
          $ s \in [ m -1 ] $.
 
Suppose that 
 $ \theta $
 is not constant on the set 
 $ [ m ] $,
  we shall show next that the left hand side of  equation (\ref{eq:main2:02}) also vanishes.
   
   If  
   $ m =2 $,
 then
 \begin{align*}
 \FNm{2}\leq | \{ ( i_1, i_2, i_3 )\in [ N ]^3:\ 
 \theta_1 \circ \varepsilon_1 ( i_1, i_2 ) 
 = 
 \theta_2 \circ \varepsilon_2 ( i_2, i_3 )  \} | .
 \end{align*}   
 Denoting $ (i,j) = \varepsilon_1 ( i_1, i_2) $ 
 and 
 $ k =i_3 $, 
 we get that
   \begin{align*}
   \FNm{2}\leq |  \big\{ ( i, j, k )\in [ N ]^3:\ 
   \theta_1 (i, j)  
   \in \{ 
   \theta_2 (i, k), (k, j)\}  \big\} | .
   \end{align*} 
   so condition (C2) and the second part of condition (C3)  give that
    $ \FNm{2} = o(N^2) $
    and, from Lemma \ref{fpq:1}(4)
    it follows that
    $ \displaystyle
     \lim_{N \rightarrow \infty} 
     \mathcal{V}_N (p, q )
     =0 $.


   Next suppose that
    $ m > 2 $. 
     Suppose first 
     $ \theta $ 
     is not constant on the set 
     $ [ m ] $
 and let
  $ s \in [ m-1 ] $
  be such that
   $ \theta_s \neq \theta_{s+1} $. 
   According to Lemma \ref{fpq:1}, we have that 
   $ \varepsilon_s \neq \varepsilon_{s+1} $
   and 
   $ s+1 \in \{ p(s), q(s) \} $.
  Denote
    $ (i, j) = \varepsilon_s( i_s, j_s ) $
    and 
    $ k = j_{s+1} $.
   With these notations, if 
    $ s+1 = p(s) $,
    then
    \begin{align*}
     \FNm{ \{ s, s+1 \} } & \leq 
     |\big\{
     (i, j, k):\ 
     \pi_1 \circ \theta_s (i, j)  = \pi_1 \circ \theta_{s+1} ( i_{s+1}, k )
     \big\}|   \\
     & \leq 
     |\big\{
          (i, j, k):\ 
          \pi_1 \circ \theta_s (i, j) 
          \in \{ \pi_1 \circ \theta_{s+1} (i, k) ,
          \pi_1 \circ \theta_{s+1} (k, j )
          \}
          \big\}|\\
     & \leq
     Y_1( \theta_s, \theta_{s+1} ).
     \end{align*}
 So the first part of condition (C3) and (C2) give that
  $ \FNm{ \{ s, s+1 \} } = o( N^3 ) $,
  i.e, according to Lemma \ref{fpq:1}(4), 
  the left-hand side of equation (\ref{eq:main2:02})
  vanishes.
  
  If 
  $ s+1 = q(s) $,
  then 
      \begin{align*}
       \FNm{ \{ s, s+1 \} } & \leq 
       |\big\{
       (i, j, k):\ 
       \pi_2 \circ \theta_s (i, j)  = \pi_2 \circ \theta_{s+1} ( i_{s+1}, k )
       \big\}|   \\
       & \leq 
       |\big\{
            (i, j, k):\ 
            \pi_2 \circ \theta_s (i, j) 
            \in \{ \pi_2 \circ \theta_{s+1} (i, k) ,
            \pi_2 \circ \theta_{s+1} (k, j )
            \}
            \big\}|\\
       & \leq
       Y_2( \theta_s, \theta_{s+1} )
       \end{align*}
and, similarly, the let-hand side of equation (\ref{eq:main2:02})
  vanishes again.

     Next, still under assumption that
           $ m $ 
           is even and that
            $ \varepsilon_s \neq \varepsilon_{ s + 1 }$
            for all 
            $ s \in [ m -1 ], $
  we shall show that equation (\ref{eq:main2:02}) holds when 
     $\theta $ 
     is constant on the set 
     $ [ m ] $.
     
     If 
    $ \theta_t = \sigma_{s, N} $ 
    for all
    $ t \in [ m ] $ 
    and
    $ m > 2 $,
     then the left hand side of (\ref{eq:main2:02}) vanishes by Remark \ref{remark:4:1}(i) and so does the right hand side, since 
     $ x_s $ 
     is a circular element of
     $ \mathfrak{A}$.
     If 
     $ m =2 $, 
     then part (ii) of Remark \ref{remark:4:1} gives that any pair
     $ p^\prime, q^\prime
      \in \cP_2^{ \varepsilon^\prime } (n - m )$ 
      such that
      $ \displaystyle 
     \lim_{ N \rightarrow \infty } \mathcal{V}_{ \overrightarrow{\sigma}^\prime, \varepsilon^\prime } ( 
      p^\prime, q^\prime ) \neq 0 $
      extends uniquely to some
      $ p, q \in \cP_2^\varepsilon ( n ) $ 
      such that
      \begin{align*}
     \lim_{N \rightarrow \infty}
     \VN = 
     \lim_{ N \rightarrow \infty }
     \mathcal{V}_{\overrightarrow{\sigma}^\prime, \varepsilon^\prime, N} ( p^\prime, q^\prime )
     = \kappa_2 \big( x_s^{\varepsilon_1}, x_s^{ \varepsilon_2}  \big)
      \lim_{ N \rightarrow \infty } 
          \mathcal{V}_{\overrightarrow{\sigma}^\prime, \varepsilon^\prime, N}
          ( p^\prime, q^\prime ), 
      \end{align*}
 where, as before,
 $ \varepsilon^\prime$ and $ \overrightarrow{\sigma}^\prime $
 are the restrictions of 
 $ \varepsilon $,
  respectively
  $\overrightarrow{\sigma}$
 to the set
 $ [ n ] \setminus [ m ] $
 and
 $ p, q $
 are defined by 
 \begin{align*}
 p(s) = \left\{
 \begin{array}{ l l }
 2, & \textrm{if } s =1\\
 1, & \textrm{if } s =2\\
 p^{ \prime}(s-m) & \textrm{if } s > m
 \end{array}
  \right.
  \textrm{ respectively }
  q(s) = \left\{
   \begin{array}{ l l }
   2, & \textrm{if } s =1\\
   1, & \textrm{if } s =2\\
   q^\prime(s-m) & \textrm{if } s > m.
   \end{array}
    \right.
 \end{align*}    
    
    Suppose then that 
    $ \theta_t =\textrm{Id} $
    for all
    $ t \in [ m ] $.
 and fix 
    $ p^\prime, q^\prime
     \in \cP_2^{ \varepsilon^\prime}(n-m) $ 
     such that
    $ \displaystyle 
        \lim_{ N \rightarrow \infty } \mathcal{V}_{\overrightarrow{\sigma}^\prime, \varepsilon^\prime, N}
                  ( p^\prime, q^\prime ) \neq 0 $.
           
     If
     $ \varepsilon_1 = 1 $, 
     then, with the notations from Remark \ref{remark:3:1},  let
     $ p, q \in \cP_2^\varepsilon ( n) $
     be given by
\begin{align*}
   p(s) = \left\{ 
   \begin{array}{l l}
   \widetilde{p}_{ m} ( s ) & \textrm{ if } s \leq m \\
   p^\prime ( s- m )  & \textrm{ if } s \geq m
   \end{array}
    \right.
    \textrm{, respectively }
    q( s ) = \left\{
     \begin{array}{l l}
      \widetilde{q}_{ m} ( s ) & \textrm{ if } s \leq m \\
       q^\prime ( s- m )  & \textrm{ if } s \geq m
       \end{array}
    \right.
\end{align*}     
   In particular,
   $ p(m) =1 $,
   so
   $ i_1 = k_1 = k_m = i_m $.
   Therefore, denoting 
   $ \overrightarrow{\textrm{Id}} = ( \textrm{Id}, \dots, \textrm{Id}) $
   and 
   $ \widetilde{\varepsilon} = (1, \ast, \dots, 1, \ast ) $, we have that
   \begin{align*}
   \lim_{ N \rightarrow \infty } \VN
   = \lim_{N \rightarrow \infty }
   \mathcal{V}_{\overrightarrow{\textrm{Id}}, \widetilde{\varepsilon}, N} ( \widetilde{p}_{m}, \widetilde{q}_{m})
   \cdot
   \mathcal{V}_{\overrightarrow{\sigma}^\prime, \varepsilon^\prime, N}
             ( p^\prime, q^\prime ).
   \end{align*}
     Using equation (\ref{eq:u:frcm:1}), the equality  above becomes
   \begin{align*}
   \lim_{ N \rightarrow \infty } \VN
   =
   \kappa_{m}(  x_1, x_1^\ast, \dots, x_1, x_1^\ast )
   \cdot
    \lim_{N \rightarrow \infty }
   \mathcal{V}_{\overrightarrow{\sigma}^\prime, \varepsilon^\prime, N}
                ( p^\prime, q^\prime ).
   \end{align*}
   Moreover, from Lemma \ref{tech:main:01} for any  
   $ p , q   \in \cP_2^\varepsilon ( n )$
   we have that
   $ \displaystyle \lim_{ N \rightarrow \infty }
   \VN= 0 $
   unless
   $ p _{| [ m ] } = \widetilde{ p}_{ m } $
   and
   $ q_{| [ m ] } = \widetilde{q}_{ m } $.
   
    The case 
    $ \varepsilon_1 = \ast $ 
    is similar. We let
   \begin{align*}
      p(s) = \left\{ 
      \begin{array}{l l}
      \widetilde{q}_{ m} ( s ) & \textrm{ if } s \leq m \\
       p^\prime ( s- m )  & \textrm{ if } s \geq m
      \end{array}
       \right.
       \textrm{, respectively }
       q( s ) = \left\{
        \begin{array}{l l}
          \widetilde{p}_{ m} ( s ) & \textrm{ if } s \leq m \\
           q^\prime ( s- m )  & \textrm{ if } s \geq m
          \end{array}
       \right.
   \end{align*} 
 and using equation (\ref{eq:u:frcm:1}) we obtain
  \begin{align*}
  \lim_{ N \rightarrow \infty } \VN
  =
  \kappa_{m}(  x_1^\ast, x_1, \dots, x_1^\ast, x_1 )
  \cdot
   \lim_{N \rightarrow \infty }
  \mathcal{V}_{\overrightarrow{\sigma}^\prime, \varepsilon^\prime, N}
               ( p^\prime, q^\prime ).
  \end{align*}  
  Finally, in this case  Lemma \ref{tech:main:01} gives that 
  that for any  
     $ p, q  \in \cP_2^\varepsilon ( n )$
     we have that
     $ \displaystyle \lim_{ N \rightarrow \infty }
     \VN = 0 $
     unless
     $ p_{| [ m ] } = \widetilde{q}_{ m } $
     and
     $ q_{| [ m ] } = \widetilde{p}_{ m } $.
 \end{proof}

\section{Examples}

 In this section we shall apply the results from the previous section to two classes of particular permutations on the entries of square matrices: the partial transpose with growing number of blocks and the mixing map. Moreover we give an example of a sequence of permutations which does not satisfy our assumptions, the partial transpose with fixed number of blocks.
 
 We define the partial transpose 
 $ \Gamma_{ b, d}  \in \mathcal{S}( [ bd ]^2 ) $
 as follows.
 Fix two positive integers 
 $ b $ and $ d $ 
 and let 
 $ M = bd $.
 First consider the bijection
 $ \varphi_{b, d}: 
 [ M]^2 \rightarrow \big( [ b ] \times [ d ] \big)^2 $
 given by
 $ \varphi_{ b, d } ( i, j) 
 = ( \alpha_1, \beta_1, \alpha_2, \beta_2 ) $
 whenever
  \begin{align*}
  (i, j) = \big(
   ( \alpha_1 - 1 )d + \beta_1, ( \alpha_2 -1 ) d + \beta_2 
   \big).
  \end{align*}
 Then define
   $ \rho: \big( [ b ] \times [ d ] \big)^2 
   \rightarrow
   \big( [ b ] \times [ d ] \big)^2 $
   by
   \begin{align*}
   \rho ( \alpha_1, \beta_1, \alpha_2, \beta_2 ) = 
   ( \alpha_1, \beta_2, \alpha_2, \beta_1).
   \end{align*}
   The $ (b, d) $-partial transpose is the permutation 
   \begin{align*}
   \Gamma_{b, d} =
    \varphi_{b, d}^{ -1} \circ \rho \circ \varphi_{ b, d }.
   \end{align*}

Intuitively, each
 $ bd \times bd $
 square matrix can be seen as a
 $ b \times b $
 whose entries are
  $ d \times d $ 
  block matrices.
  Its 
  $ (b, d) $-partial transpose is obtained by taking the matrix transpose of each of the 
  $ b^2 $ 
  blocks, but keeping the position of the blocks.

  With the notations above, we have the following result.
  
  \begin{example}\label{thm:partialtranspose}
   Suppose that
    $ ( b_N )_N $ 
   and 
   $ (d_N )_N $ 
   are two non-decreasing sequences of positive integers such that
   $ \displaystyle
   \lim_{ N \rightarrow \infty} b_N 
   = \lim_{ N \rightarrow \infty } d_N = \infty $.
   Then
   $ U_N^{ \Gamma_{b_N, d_N}}$
   is asymptotically circular and free from
   $ U_N $.
  \end{example}
  \begin{proof}
   By Theorem \ref{thm:main:2} and Remark \ref{c:c1c2},
    we only need to show that 
   $ \big( \Gamma_{b_N, d_N}\big)_N $
   satisfies  condition
   (C).
  For any positive integers 
  $ b, d $
  we have that
   \begin{align*}
   Y_{1, bd} (\Gamma_{b, d}, \Gamma_{b, d})
   & = | \big\{
    (i, j, k) \in [ bd ]^3:\ 
  \pi_1 \circ \Gamma_{b, d} ( i, j)
    \in\{  \pi_1 \circ \Gamma_{b, d}( i, k), \pi_1 \circ \Gamma_{b, d}(k, j) \}
    \big\}|\\
    = &
    |\big\{
    (\alpha_s, \beta_s)_{ 1 \leq s \leq 3 } \in
     ( [b] \times [ d ])^3 :\
     ( \alpha_1, \beta_2 ) \in \{ ( \alpha_1, \beta_3 ), ( \alpha_3, \beta_2 )\}
    \big\}|\\
    = &
    |\big\{
        (\alpha_s, \beta_s)_{ 1 \leq s \leq 3 } \in
         ( [b] \times [ d ])^3 :\
         \beta_2 = \beta_3 \textrm{ or  } \alpha_1 = \alpha_3 
        \big\}|\\
       = & b^3 d^3 \left( \frac{1}{b} + \frac{1}{d} \right).
   \end{align*}
   
   Similarly,
   $ \ds 
    Y_{2, bd} ( \Gamma_{ b, d }, \Gamma_{ b, d} )
    = b^3 d^3 \big( \frac{1}{b} + \frac{1}{d} \big) $.
    
    Hence,
    \begin{align*}
   \lim_{ N \rightarrow \infty }
   ( b_Nd_N )^{ -3} Y_N ( \Gamma_{b_N, d_N} ) = \lim_{ N \rightarrow \infty }\frac{2}{b_N} + \frac{2}{d_N} = 0.
    \end{align*}
  \end{proof}
  
  With the notations from above, 
  let
  $ \mu_N : [ N ]^4 \rightarrow [ N ]^4 $
  given by
  $ \mu_N ( a, b, c, d ) = ( a, c, b, d ) $.
  We define the `mixing map' 
  $ m_N : [ N^2 ]^2 \rightarrow [ N^2 ]^2 $
  via (see \cite{MandLinowZycz}) :
  \begin{align*}
  m_N = \varphi_{ N, N}^{ -1} \circ \mu_N \circ \varphi_{N , N} .
  \end{align*}
 As in the case of the partial transpose, we have the following result.
 
 \begin{example}\label{thm:mixingmap}
 The random matrix
 $ U_N^{ m_N} $
 is asymptotically (as 
 $ N \rightarrow \infty $) 
 circularly distributed; furthermore, the family 
 $\{ U_{ N^2 }, U_{ N^2 }^{ \Gamma_{N, N}},
  U_{ N^2 }^{\mu_N}\} $
 is asymptotically free. 
 \end{example}
    
\begin{proof}
It is straightforward to check that  
   $ ( m_N )_N $
   satisfies  condition
    (C).
We have that 
     \begin{align*}
  Y_1(m_N, m_N  )  &   = 
  | \big\{
  ( i, j, k ) \in [ N^2 ]^3: \pi_1 \circ m_N ( i, j)
  \in\{  \pi_1 \circ m_N ( i, k), \pi_1 \circ m_N (k, j) \}
  \big\}| \\
  &  =  
  | \{
  (\alpha_1, \beta_1, \alpha_2, \beta_2, \alpha_3, \beta_3) \in [ N ]^6: (\alpha_1, \alpha_2) \in
   \{ ( \alpha_1, \alpha_3),   (\alpha_3, \alpha_2 ) \}
  \big\} |\\
  &  = 
    | \{
    (\alpha_1, \beta_1, \alpha_2, \beta_2, \alpha_3, \beta_3) \in [ N ]^6:  \alpha_3 \in
     \{  \alpha_1, \alpha_2\}
    \big\} |\\
  & =  2 N^5,
     \end{align*}
 while, similarly
      \begin{align*}
   Y_2 ( m_N, m_N ) &   = 
   | \big\{
    ( i, j, k ) \in [ N^2 ]^3: \pi_2 \circ m_N ( i, j)
    \in\{  \pi_2 \circ m_N ( i, k), \pi_2 \circ m_N (k, j) \}
    \big\}| \\
     = & 
     | \big\{
  (\alpha_1, \beta_1, \alpha_2, \beta_2, \alpha_3, \beta_3) \in [ N ]^6:\
  (\beta_1, \beta_2) \in 
 \{ 
 (\beta_1, \beta_3), ( \beta_3, \beta_2)
  \}
    \big\} |\\
    &  = 
        | \{
        (\alpha_1, \beta_1, \alpha_2, \beta_2, \alpha_3, \beta_3) \in [ N ]^6:  \beta_3 \in
         \{  \beta_1, \beta_2\}
        \big\} |
      =  2 N^5,      
         \end{align*} 
   hence
   $ \displaystyle 
   \lim_{N \rightarrow \infty }
    \frac{1}{(N^2)^3} Y ( m_N, m_N ) = 0, 
   $.
   
   It remains to check that the pair
    $(m_N, \Gamma_{N, N})$ 
    satisfies
     (C3).
     The computations are very similar to the ones above. That is
    \begin{align*}
    X_N(\mu_N, \Gamma_{N, N})
    & =| \big\{ (\alpha_s, \beta_s)_{1\leq s \leq 3} \in [ N ]^6 :\ \\
     & \hspace{4cm}
    ( \alpha_1, \alpha_2, \beta_1, \beta_2)
     \in 
     \{ 
     (\alpha_1, \beta_3, \alpha_3, \beta_1),
     (\alpha_3, \beta_2, \alpha_2, \beta_3)\}\}
    \big\} |\\
    & = 2N^3 = o(N^4),
    \end{align*}
 and also
    \begin{align*}
    Y(m_N, \Gamma_{N, N})
     & =| \big\{ (\alpha_s, \beta_s)_{1\leq s \leq 3} \in [ N ]^6 :\ 
     (\alpha_1, \alpha_2) 
     \in \{
      ( \alpha_1, \beta_3), (\alpha_3, \beta_2)
      \}\\
      &\hspace{7cm}\textrm{ or } ( \beta_1, \beta_2) 
     \in \{ 
     ( \alpha_3, \beta_1), ( \alpha_2, \beta_3)
     \}
     \big\} \\
     &  = o(N^6).
    \end{align*}
\end{proof}   

 With the notations from the beginning of this section, the matrix transpose can be seen as  the 
$ (1, N) $-partial transpose. 
It does not satisfy (C1), as it transforms any two entries from the same column into entries from the same row. Yet, as shown in \cite{MingoPopaUnitarlyInvariant2016}, 
$U_N $ 
and 
$U_N^t $
are asymptotically free. 
More generally, the relation (C1) is not satisfied by
$ \Gamma_{b, N} $ 
whenever 
$ b $ 
is a fixed and
$ N \rightarrow \infty $,
since
the computation from the proof of Example \ref{thm:partialtranspose} gives that
\begin{align*}
(bN)^{ -3} Y^{(r)}_N ( \Gamma_{b, N}) = \frac{1}{b}\neq 0.
\end{align*} 
As stated in the Introduction, partial transposes with fixed numbers of blocks have been previously studied in the framework of Wishart random matrices (as in \cite{BaNec} and\cite{MingoPopaWishart1}). We will describe some properties of permuted Haar unitaries by partial transposes with fixed blocks in the example below.
\begin{example}\label{d-part-transp}
   Let
    $ b $ 
    be a (fixed) positive integer.
    As in Section 5, we denote by 
    $ \Gamma_{b, N} $
    the partial transpose of parameters 
    $ b $ and $ N $
    on $ bN \times bN $ 
    matrices.
    
     We shall show that the asymptotic (as 
     $ N \rightarrow \infty $)
      $ \ast $-distribution of
    $ U^{ \Gamma_{b, N}}_{bN} $ 
    is R-diagonal with the alternating free cumulants of order
    $ 2n $ 
     equal to 
     $ b^{ 2-2n}  ( -1 )^{ n -1} C_{ n -1 } $ 
     and that
     $ U_{bN} $
     and
     $ U^{ \Gamma_{b, N}}_{bN} $ 
     are asymptotically free.
     In particular, the limit
     $ \ast$-distribution of
     $ b \cdot U_{ bN}^{ \Gamma_{b, N}} $
     equals that of a sum of
     $ b^2 $
     free Haar unitaries.
\end{example}
\begin{proof}
   For any positive integers 
   $ b, d $
   and 
   $ N = bd $,
   we have that
   \begin{align*}
   X_N(\Gamma_{b, d}. \Gamma_{b, d})& = 
  | \{ ( i_1, i_2)  \in [ bd ]^2: 
   \pi_1 \circ \Gamma_{ b, d } ( i_1, i_2) = i_1 \textrm{ or }
   \pi_2 \circ \Gamma_{ b, d } (i_1, i_2 ) = i_2 \} |\\
   = &
   | \{ 
   ( \alpha_1, \beta_1, \alpha_2, \beta_2)  \in 
   \big( [ b] \times [ d ]  \big)^2  : 
    ( \alpha_1, \beta_1) = ( \alpha_1, \beta_2 ) \textrm{ or }
    ( \alpha_2, \beta_2 ) = ( \alpha_2, \beta_1 )
   \} |\\
    = & 
     | \{ 
     ( \alpha_1, \beta_1, \alpha_2, \beta_2)  \in 
     \big( [ b] \times [ d ]  \big)^2  :  \beta_1 = \beta_2 \}|
     = b^2 d . 
   \end{align*} 
  So
  $ \big(  \Gamma_{b, N} \big)_N $ 
  satisfies condition 
  (C2)
  since
  \begin{align*}
  \lim_{N \rightarrow \infty }\frac{ X_{bN} ( \Gamma_{b, N})}{ ( bN)^2 } = \lim_{ N \rightarrow \infty } N^{ -1} = 0.
  \end{align*}
 
         As before, let 
     $ n $ 
     be a positive integer,
     $ \epsilon \rightarrow \{ 1, \ast \} $
     and
     $ \theta
      \rightarrow \{ \textrm{Id}_{bN} , \Gamma_{ b, N }\} $.
     Since both alternate free cumulants of order 
     $ 2r $ 
     of a Haar unitary equal 
     $ (-1)^{ r -1} C_{ r -1} $,
     it suffices to show the following  statement. 
     
     \emph{Fix 
      $ m \leq n $, 
     and denote by 
     $ \overrightarrow{\sigma}^\prime$
    and
    $ \varepsilon^\prime $
    the restrictions of
    $\overrightarrow{\sigma} $,
    respectively 
  $ \varepsilon$
  to
  $ [ n ] \setminus [ m ] $.
  Also, fix 
  $ \tau $ 
  a non-crossing partition on 
  $ [ n ] \setminus [ m ] $
  and let
  $ \omega = [ m ] \oplus \tau $. Then
   }
   \begin{align*}
   \lim_{ N \rightarrow \infty} 
   \sum_{ \substack{ p, q \in \cP_2^{ \epsilon }( n )\\
   p \vee q  = [ m ] \oplus \tau } }
   \mathcal{V}_{\overrightarrow{\sigma}, \varepsilon, bN }{(p,q)}
   = k(\varepsilon, \theta, m ) \cdot
   \lim_{ N \rightarrow \infty }
    \sum_{ \substack{ p^\prime, q^\prime 
      \in \cP_2^{ \epsilon^\prime }( n - m ) \\
      p^\prime \vee q^\prime = \tau } }
      \mathcal{V}_{\overrightarrow{\sigma}^\prime, \varepsilon^\prime, bN}
                   ( p^\prime, q^\prime )
   \end{align*}
   \emph{where}
   \begin{align*}
   k(\varepsilon, \theta, m ) = \left\{
   \begin{array}{l}
    (-1)^{\frac{m}{2} -1} C_{ \frac{m}{2}-1 }, \
    \textrm{if $m $ is even, $ \epsilon_{| [ m ] } $ is alternating,
     $ \theta_{| [ m ] } = \textrm{Id}_{bN} $ }\\
     b^{ 2 - m }
     \cdot
     (-1)^{\frac{m}{2} -1} C_{ \frac{m}{2}-1 }, \ 
     \textrm{if $m $ is even, $ \epsilon_{| [ m ] } $ is alternating,
          $ \theta_{| [ m ] } = \Gamma_{b, N} $ }
     \\
     0, \ \textrm{ otherwise.}
   \end{array}
   \right.
   \end{align*}
   
   If 
   $ m $
    is not even, then the statement is vacuously true, since 
   $ p \vee q $
   has only blocks with even number of elements. 
   If
   $ \epsilon $ 
   is not alternating on 
   $ [ m ] $, 
   then the result follows trivially from Lemma 
   \ref{lemma:c2}. 
   If 
   $ \theta $
   is not constant on 
   $ [ m ] $, 
   since the sequence
   $ ( \Gamma_{ b, N})_N $
   satisfies condition
   (C2)
   we have that the limit in the left hand side cancels, 
   hence the assertion holds true.
   Moreover, if 
   $ m $ 
   is even,
   $ \epsilon $
   is alternating on 
   $ [ m ] $
   and
   $ \theta $ 
   is constant
   $ \textrm{Id}_N $
   on $ [ m ] $, 
   then again the result follows trivially from equations (5) and (6) in Section 3.
   
   It remains to show the property for 
   $ m $
   even, 
   $ \epsilon $
   alternating on 
   $ [ m ] $
   and
   $ \theta $
   constant 
   $ \Gamma_{ b, N } $
   on 
   $ [ m ] $. 
   In this framework, with the notations from Remark \ref{remark:3:1},  Lemma \ref{tech:main:01} gives that the restrictions of 
   $ p $, 
   respectively
   $ q $
   to the set
   $ [ m ] $
   are either
   $ \widetilde{p}_{m} $, respectively $ \widetilde{q}_{m} $
   or
   $ \widetilde{q}_{m} $, respectively $ \widetilde{p}_{m} $.
   (Remember that 
   $ \widetilde{p}_{m} $ 
   and 
   $ \widetilde{q}_{m} $
   are given by
   $ \widetilde{p}_{m} (m) =1 $
   and
   $ \widetilde{p}_{m} (2t ) = 2t + 1 $
   respectively
   $ \widetilde{q}_{m} (2t) = 2t - 1 $
   for 
   $ t = 1, 2, \dots, \displaystyle \frac{m}{2} $).
   
   For each 
   $ s = 1, 2, \dots, m $,
   write
    $ i_s  = ( \alpha_s -1) N + \beta_S $
    with
    $ \alpha_s \in [ b ] $
    and
    $ \beta_s \in [ N ] $.
   
   Suppose first that
    $ \epsilon_1 = 1 $. 
    Then
    \begin{align*}
    (k_s, l_s) = \left\{
    \begin{array}{ll}
   \big( ( \alpha_s -1 ) N + \beta_{s + 1} , ( \alpha_{ s + 1} - 1 )N + \beta_s \big)&
     \textrm{ if $ s $ is odd }\\
     \big( (\alpha_{ s + 1} - 1 ) N + \beta_s ), ( \alpha_s -1)N + \beta_{s+1}
     \big) & 
     \textrm{ if $ s $ is even. }
    \end{array}
    \right.
    \end{align*}
  
  If
   $ p_{ | [ m ] } = \widetilde{p}_{m} $ 
  and
   $ q_{ | [ m ] } = \widetilde{q}_{m} $
   then the conditions
   $ k_s = k_{ p(s) } $ 
   and 
   $ l_s = l_{ q (s ) } $
   become
    \begin{align*}
    \left\{
   \begin{array}{l l}
    \alpha_1 = \alpha_{ m + 1} \\
    \beta_{s} = \beta_{ s + 2 } \textrm{ for each } $ s = 1, 2, \dots, m-1 $ 
    \end{array} 
     \right.
    \end{align*}
    hence
     $ \mathcal{A}_{\overrightarrow{\sigma}, \varepsilon, bN }^{(p,q)} \leq b^m N^{ \frac{m}{2} } =o\big((bN)^{\frac{m}{2}} \big),
     $
     therefore
     $ \displaystyle 
      \lim_{ N \rightarrow \infty } \mathcal{V}_{\overrightarrow{\sigma}, \varepsilon, bN }{(p,q)} = 0 $.
      
   If
     $ p_{ | [ m ] } = \widetilde{q}_{m} $ 
    and
     $ q_{ | [ m ] } = \widetilde{p}_{m } $
    the conditions
     $ k_s = k_{ p(s) } $ 
     and 
     $ l_s = l_{ q (s ) } $
     become    
        \begin{align}\label{b:catalan:1}
         \left\{
        \begin{array}{l l}
          \alpha_{s} = \alpha_{ s + 2 } \textrm{ for each } $ s = 1, 2, \dots, m-1 $ \\
          \beta_1 = \beta_{ m + 1}. 
         \end{array} 
          \right.
         \end{align} 
  In particular, since 
  $ m $
  is even,
  $ \alpha_1 = \alpha_{m+1} $,
  therefore
  \begin{align*}
   i_1 = ( \alpha_1 -1)N + \beta_1 = 
   ( \alpha_{ m+1} -1) N + \beta_{ m +1 } = i_{ m + 1}.
   \end{align*} 
   So, if
       $ p = \widetilde{q}_{m} \oplus p^\prime $
       and
       $ q = \widetilde{p}_{m} \oplus q^\prime $
       for some
     $ p^\prime, q^\prime \in \cP_2^{ \epsilon^\prime } ( n - m )$
  we have that
  each $(m-n)$-tuple
     $ (i_s, j_s)_{ m+1 \leq s \leq n } $ 
     from
     $ \mathcal{A}_{\overrightarrow{\sigma}^\prime, \varepsilon^\prime, bN}^{ ( p^\prime, q^\prime )} $
  uniquely determines 
  $ \alpha_1 $
  and
  $ \beta_1 $  
  (since
  $ i_{ m +1 } $, 
  that is 
  $i_1$,
   is given). Hence we have that
  \begin{align*}
  | \mathcal{A}_{\overrightarrow{\sigma}, \varepsilon, bN}^{ (p, q)}| = &
  |\{ ( \alpha_s, \beta_s)_{ 2\leq s \leq m } : \ \alpha_s, \beta_s \textrm{ satisfy conditions (\ref{b:catalan:1})}\}  | 
  \cdot
  | \mathcal{A}_{ \overrightarrow{\sigma}^\prime, \varepsilon^\prime, bN}^{ ( p^\prime, q^\prime )}|\\
   = &
   b N^{ m -1} \cdot
     | \mathcal{A}_{ \overrightarrow{\sigma}^\prime, \varepsilon^\prime, bN}^{ ( p^\prime, q^\prime )}|.
  \end{align*}
  
We will need some more details on the unitary Weingarten function,
 $ \Wg_N $, introduced in Section \ref{section:background}. First, since for fixed
  $ p, q $ 
  pair partitions on the set
  $ [ n ] $,
  and
  $ N > n $, 
  the map 
  $ N \mapsto \Wg_N (p, q)$ is a rational function,
  let us further extend the notations from Remark \ref{remark:2.5} and write
   \begin{align*}
    \Wg_N ( p, q ) = C_{ p, q } \cdot N^{ - n + | p \vee q | } + D_{ p, q} \cdot N^{ - n + | p \vee q | -2 }  + o (   N^{ - n + | p \vee q | - 2 }  ),
   \end{align*}
 where $p \vee q = \{B_1, \dots, B_k\}$, 
 $\ds C_{p,q} = \prod_{i = 1}^{k} (-1)^{|B_i|/2 -1} C_{|B_i|/2 -1}$, and
  $ D_{ p, q} $ 
  depends only on 
  $ p $ and $ q $.
    
With this notation, if 
 $ p \vee q = [ m ] \oplus (p^\prime \vee q^\prime)  $
 then
 $ | p \vee q | = |p^\prime \vee q^\prime| + 1 $
 and
 \begin{align*}
 C_{p, q} = (-1)^{ \frac{m}{2} -1} C_{ \frac{m}{2} -1} 
 \cdot
 C_{ p^\prime, q^\prime}. 
 \end{align*}  
 Therefore, with the notations from Section 3,   
\begin{eqnarray*}
 \mathcal{V}_{\overrightarrow{\sigma}, \varepsilon, N }{(p,q)}
  & = &  
  \Wg_{ bN} (p, q) \cdot \frac{1}{b N} | \mathcal{A}_{\overrightarrow{\sigma}, \varepsilon, bN}^{ ( p, q) } | \\
   & = &
   ( bN )^{ - n + | p \vee q | } \cdot C_{ p, q}  \frac{1}{b N} | \mathcal{A}_{\overrightarrow{\sigma}, \varepsilon, bN}^{ ( p, q) } |
   + O ( N^{ -2} ) \\
   & = &
   (-1)^{ \frac{m}{2} -1} C_{ \frac{m}{2} -1} 
   \cdot
   C_{ p^\prime, q^\prime } \cdot 
   ( b N )^{ - m + 1 } \cdot
    ( b N )^{ - ( n - m ) + |p^\prime \vee q^\prime| } 
     \frac{1}{bN} \cdot b N^{ m -1} 
    | \mathcal{A}_{\overrightarrow{\sigma}^\prime, \varepsilon^\prime, bN}^{ ( p^\prime, q^\prime ) }  | \\
    && \mbox{} 
    + O( N^{ -2} )\\
    & = &
    \big[ b^{ 2- m } \cdot
    (-1)^{ \frac{m}{2} -1} 
    C_{ \frac{m}{2} -1} 
    \big]
    \cdot
    \big[C_{ p^\prime, q^\prime } ( b N )^{ - ( n - m ) + |p^\prime \vee q^\prime| }
    \cdot \frac{1}{bN}
    | \mathcal{A}_{\overrightarrow{\sigma}^\prime, \varepsilon^\prime, bN}^{ ( p^\prime, q^\prime ) }  |
    \big] + O(N^{ -2 } )\\
    & = &
    \big[ b^{ 2- m } \cdot
        (-1)^{ \frac{m}{2} -1} 
        C_{ \frac{m}{2} -1} 
        \big]
        \cdot
        \mathcal{V}_{\overrightarrow{\sigma}^\prime, \varepsilon^\prime, bN} ( p^\prime, q^\prime )
        + O(N^{ -2 } ),
 \end{eqnarray*}
 
In the case
 $ \epsilon_1 = \ast $,
 same argument gives that, if 
 $ p = \widetilde{q}_{m} \oplus p^\prime $
 and
 $ q = \widetilde{p}_{m} \oplus q^\prime $
 for some
 $ p^\prime, q^\prime
  \in \cP_2^{ \widetilde{ \epsilon } } ( n - m ) $, 
 then
 $ \displaystyle
 \lim_{N \rightarrow \infty } 
 \mathcal{V}_{\overrightarrow{\sigma}, \varepsilon, bN }{(p,q)} = 0 $,
 while if
 $ p = \widetilde{p}_{m} \oplus p^\prime $
 and
 $ q = \widetilde{q}_{m} \oplus q^\prime $,
 then
 \begin{align*}
 \mathcal{V}_{\overrightarrow{\sigma}, \varepsilon, bN }{(p,q)}
  = \big[ b^{ 2 - m } \cdot
          (-1)^{ \frac{m}{2} -1} 
          C_{ \frac{m}{2} -1} 
          \big]
          \cdot
          \mathcal{V}_{\overrightarrow{\sigma}^\prime, \varepsilon^\prime, bN} ( p^\prime, q^\prime )
          + O(N^{ -2 } ),
 \end{align*}
and the proof is complete.
\end{proof}

\section{Second order fluctuations and almost sure convergence}

In this section we prove Theorem \ref{thm:12} from the introduction, the result of this section shows at the same time that convergence from Theorem \ref{thm:14} can be also stated in the almost sure sense. We employ the technique explained in detail in Chapters 4 and 5 of \cite{MingoSpeicher}. Namely we shall show that

 \begin{multline}\label{eqn:covBound}
\limsup_{ N \rightarrow \infty } 
\Big|\cov \left( \Tr ( 
U_N^{ ( \sigma_{1, N}, \varepsilon_1 ) }
\cdots
U_N^{ ( \sigma_{m, N}, \varepsilon_m ) }
), 
\Tr (
U_N^{ ( \sigma_{ m + 1, N}, \varepsilon_{ m + 1} ) }
\cdots
U_N^{ ( \sigma_{ m + r, N}, \varepsilon_{ m + r} ) }
) 
\right) \Big|  \\
< C (m, r, \varepsilon)
\end{multline}
for any permutations $\sigma_{1,N},\ldots,\sigma_{m+r,N}\in \mathcal{S}\left([N]^2\right)$, $ \varepsilon : [ n ] \rightarrow \{ 1, \ast \} $
and
$ \theta : [ n ] \rightarrow \{ \textrm{Id}, \sigma_N \} $.
Then with notations as in Section 4 let 
\begin{align*}
Z_n = \frac{1}{N}\Tr \big(
U^{ ( \theta_1, \varepsilon_1 ) }
U^{ ( \theta_2, \varepsilon_2 ) }
\cdots
U^{( \theta_n, \varepsilon_n )}
\big).
\end{align*}
We get for any $\epsilon>0$
\begin{align*}
\P\left(\left|Z_n-\E (Z_n)\right|>\epsilon\right)<\frac{\mathrm{Var}\left(\Tr \big(
	U^{ ( \theta_1, \varepsilon_1 ) }
	U^{ ( \theta_2, \varepsilon_2 ) }\cdots
	U^{( \theta_n, \varepsilon_n )}
	\big)
	\right)}{\epsilon^2 N^2}
\end{align*}
and $\eqref{eqn:covBound}$ gives us that $\mathrm{Var}\left(\Tr \big(
U^{ ( \theta_1, \varepsilon_1 ) }
U^{ ( \theta_2, \varepsilon_2 ) }\cdots
U^{( \theta_n, \varepsilon_n )}
\big)
\right)$ stays bounded.
Since $\E (Z_N)$ converges by Theorem \ref{thm:11}, then we get that $Z_N$ converges almost surely to the same limit almost surely, by a simple application of Borel-Cantelli Lemma. This proves Theorem \ref{thm:12}. Thus our only job in this section is to prove \eqref{eqn:covBound}.

 Let 
$ m, r $ 
be two positive even integers and 
denote by
$ \varepsilon^\prime $,
respectively  by
$ \varepsilon^{ \prime\prime }$
the restrictions of a map
$ \varepsilon: [ m + r ] \rightarrow \{ 1, \ast \} $
to the sets
 $ [ m ] $,
 respectively
 $ [  m  + r ] \setminus [ r ] $
 (as before, we shall write
 $ \varepsilon_s $
 for 
 $ \varepsilon(s) $).
 If 
 $ p_1 \in \cP_2^{ \varepsilon^\prime}(m) $
 and
 $ p_2\in \cP_2^{ \varepsilon^{ \prime\prime}}(r) $,
 we will denote by
 $ p_1 \oplus p_2 $
 the pair partition on 
 $ [ m + r ] $ 
 given by
 \begin{align*}
 p_1 \oplus p_2 (s) = 
 \left\{
 \begin{array}{ll}
 p_1(s) & \textrm{ if } s \leq m \\
 p_2 ( s -m ) & \textrm{ if } s > m.
 \end{array}
 \right.
 \end{align*}
 
 Remark that if
 $ p_1, q_1 \in \cP_2^{ \varepsilon^\prime}(m) $
 and
 $ p_2, q_2 \in \cP_2^{ \varepsilon^{ \prime\prime}}(r) $
 then
 $ (p_1 \oplus p_2) \vee  (q_1 \oplus q_2)   = 
 p_1 \vee q_1 \oplus p_2 \vee q_2 $,
 i.e.
$ p_1 \oplus p_2 \vee q_1 \oplus q_2 $
is the juxtaposition of
$ p_1 \vee q_1$  and $ p_2 \vee q_2 $.
In particular, since
the first coefficient
 $ C_{p, q} $ 
 of
  $\Wg_N (p, q) $
   depends only on the lengths of the cycles of
   $ p \vee q $,
   we have that
   \begin{align*}
   C_{ p_1\oplus p_2, q_1 \oplus q_2} = C_{ p_1, q_1}
    \cdot C_{p_2, q_2}.
\end{align*}

Adapting the techniques from \cite{MingoPopaOrth} to the present framework, we have the following result.

\begin{lemma} \label{lemma:cov}
Let 
$ m, r $ 
be two positive integers. 
Given a map
$ \varepsilon : [ m + r ] \rightarrow \{ 1, \ast \} $, 
there exist  some 
$ C(m, r, \varepsilon )$
such that for any permutations
 $\{ \sigma_{s, N}: s \in [ m + r ], N \in \mathbb{N} \} $ 
 with 
 $ \sigma_{s, N } \in \mathcal{S}( [ N ] ^2 ) $
 we have that
 \begin{align*}
  \limsup_{ N \rightarrow \infty } |
 \cov \left( \Tr ( 
 U_N^{ ( \sigma_{1, N}, \varepsilon_1 ) }
 \cdots
 U_N^{ ( \sigma_{m, N}, \varepsilon_m ) }
   ),
 \Tr (
 U_N^{ ( \sigma_{ m + 1, N}, \varepsilon_{ m + 1} ) }
 \cdots
 U_N^{ ( \sigma_{ m + r, N}, \varepsilon_{ m + r} ) }
  ) |
 \right) 
 < C (m, r, \varepsilon).
 \end{align*}
\end{lemma}
  
  \begin{proof}
 In order to detail the manipulation of the covariance from the statement of the Theorem, we will need some additional notations, besides the notations from Section 3.
 
 First, denote by 
 $ \overrightarrow {\sigma_N}^\prime
  =
  ( \sigma_{ 1, N}, \sigma_{2, N}, \dots, \sigma_{ m, N}) $,
   $ \overrightarrow {\sigma_N}^{\prime\prime }
    =
    ( \sigma_{ m + 1, N}, \sigma_{ m+ 2, N},
     \dots, \sigma_{ m + r, N}) $
     and by
  $ \varepsilon^\prime $, 
  respectively 
  $ \varepsilon^{ \prime\prime} $
  the restriction of
  $ \varepsilon $
  to
  $ [ m ] $,
  respectively to
  $ [ m  + r ] \setminus [ r ] $;
  let
  \begin{align*}
  W_1(N, \overrightarrow {\sigma_N}^\prime, \varepsilon^\prime)
  & 
  =  \Tr ( 
   U_N^{ ( \sigma_{1, N}, \varepsilon_1 ) }
   U_N^{ ( \sigma_{2, N}, \varepsilon_2 ) }
   \cdots
   U_N^{ ( \sigma_{m, N}, \varepsilon_m ) }
     ),\\
     W_2 (N, \overrightarrow {\sigma_N}^{\prime\prime }, \varepsilon^{ \prime\prime} )
     & = \Tr (
      U_N^{ ( \sigma_{ m + 1, N}, \varepsilon_{ m + 1} ) }
      \cdots
      U_N^{ ( \sigma_{ m + r, N}, \varepsilon_{ m + r} ) }
       ).
  \end{align*}
 
 For 
 $ p, q \in \cP_2^{ \varepsilon} ( m + r ) $, 
 define
 \begin{align*}
\mathcal{B}_{\overrightarrow{\sigma_N}, \varepsilon, N }^{ (m, r)}(p, q) = \big\{  
(i_1, j_1, & \dots, i_{ m + r}, j_{ m+r}) \in [ N ]^{ 2( m + r )}
:\
 j_s = i_{ s + 1} \textrm { for }
s \notin \{ m, m+r \} \\
 &
 j_{ m } = i_1, j_{ m + r} = i_r,
 \textrm { and } k_t = k_{ p (t) }, l_s = l_{ q ( t )}
 \textrm{ for all } t \in [ m + r ]
\big\}
\end{align*} 
In particular, 
 \begin{align*}
 \mathcal{B}_{\overrightarrow{\sigma_N}, \varepsilon, N }^{(m, r)} ( p_1 \oplus p_2, q_1 \oplus q_2)
 = 
 \mathcal{A}_{\overrightarrow{\sigma_N}^\prime, \varepsilon^\prime, N}^{(p_1, q_1)}
  \cdot
 \mathcal{A}_{\overrightarrow{\sigma_N}^{\prime\prime}, \varepsilon^{\prime\prime}, N}^{(p_2, q_2)}.
 \end{align*}
 With these notations,  formula (\ref{weigarten}) gives that 
 $ \cov   ( W_1(N, \overrightarrow {\sigma_N}^\prime, \varepsilon^\prime), 
  W_2 (N, \overrightarrow {\sigma_N}^{\prime\prime }, \varepsilon^{ \prime\prime} ) $
  equals
 \begin{align*}
 \sum_{ p, q \in \cP_2^{\varepsilon} ( m + r ) } \kern-1em
 \Wg_N( p, q ) \cdot
 | \mathcal{B}_{\overrightarrow{\sigma_N}, \varepsilon, N }^{ ( m, r)} ( p, q ) | \ 
  - \mbox{} \kern-1em
  \sum_{
  \substack{
  p^\prime, q^\prime \in \cP_2^{ \varepsilon^{ \prime }} ( m )  \\
  p^{ \prime\prime}, q^{ \prime \prime } \in
    \cP_2^{ \varepsilon^{ \prime \prime}}( r )}
    } \kern-1em
  \Wg_N( p^\prime, q^\prime ) 
  \cdot
    \Wg_N( p^{ \prime\prime}, q^{ \prime\prime})
  \cdot | \mathcal{A}_{\overrightarrow{\sigma_N}^{\prime}, \varepsilon^{\prime}, N}^{(p^\prime, q^\prime)}|
  \cdot | \mathcal{A}_{\overrightarrow{\sigma_N}^{\prime\prime}, \varepsilon^{\prime\prime}, N}^{(p^{ \prime\prime}, q^{ \prime\prime } )} |.
 \end{align*}
 
 Remark that 
 $ \cP_2^{\varepsilon}(m+r) $ 
 is the disjoint union of the sets
  \begin{align*}
  \cP_2^{\varepsilon}( m, r) = \{ p \in \cP_2^{\varepsilon}(m+r):\ 
   p(s) > m  \textrm{ for some } s \leq m \}
    \end{align*}
    and
 \begin{align*}
 \{ p_1 \oplus p_2:\ p_1 \in \cP_2^{\varepsilon^\prime}(m), \
 p_2 \in \cP_2^{\varepsilon^\prime\prime} (r)  \}
 \end{align*}   
    so the summation above equals
    \begin{align*}
    \sum_{ p, q \in \cP_2^{\varepsilon}( m , r)} 
    &
  \Wg_N( p, q ) \cdot
  | \mathcal{B}_{\overrightarrow{\sigma_N}, \varepsilon, N }^{ ( m, r)} ( p, q ) |   \ 
  +\\
  & \kern-2em \sum_{
    \substack{
    p_1, q_1 \in \cP_2^{ \varepsilon^{ \prime }} ( m )  \\
    p_2, q_2 \in
      \cP_2^{ \varepsilon^{ \prime \prime}}( r )}
      }
   \Big[   
      | \mathcal{A}_{\overrightarrow{\sigma_N}^{\prime}, \varepsilon^{\prime}, N}^{(p_1, q_1)}|
          \cdot | \mathcal{A}_{\overrightarrow{\sigma_N}^{\prime\prime}, \varepsilon^{\prime\prime}, N}^{(p_2, q_2) } |
  \big(
  \Wg_N ( p_1 \oplus p_2, q_1 \oplus q_2)
  - 
    \Wg_N( p_1, q_1 ) 
    \cdot
      \Wg_N( p_2, q_2)
   \big)
    \Big].
    \end{align*}
    
 First, fix
 $ p_1, q_1 \in \cP_2^{ \varepsilon^\prime}( m) $
 and
 $ p_2, q_2 \in \cP_2^{ \varepsilon^{ \prime\prime}}(r) $
 and
 let
 $ p = p_1 \oplus p_2 $,
 $ q =  q_1 \oplus q_2 $. 
   Then
   \begin{align*}
   \Wg_N ( p, q ) &
    - 
      \Wg_N( p_1, q_1 ) 
      \cdot
        \Wg_N( p_2, q_2)=\\
     &   \big( D_{ p, q } -
        C_{ p_1, q_1} D_{ p_2. q_2} - D_{p_1, q_1} C_{p_2, q_2})
        \cdot
   N^{- ( m + r ) + |p \vee q | -2}
        + o \big( 
   N^{- ( m + r ) + |p \vee q | -2}     
         \big). 
   \end{align*} 
    Since, as seen in Section 3, 
    $ | \mathcal{A}_{\overrightarrow{\sigma_N}^{\prime}, \varepsilon^{\prime}, N}^{(p_1, q_1)}| \leq N^{ m + 1 - |p_1 \vee q_1 |} $
    and
  $ | \mathcal{A}_{\overrightarrow{\sigma_N}^{\prime\prime}, \varepsilon^{\prime\prime}, N}^{(p_2, q_2)}| \leq N^{ r + 1 - |p_2 \vee q_2 |} $,
  we have that
  \begin{align*}
  | \mathcal{A}_{\overrightarrow{\sigma_N}^{\prime}, \varepsilon^{\prime}, N}^{(p_1, q_1)}|
    \cdot 
 | \mathcal{A}_{\overrightarrow{\sigma_N}^{\prime\prime}, \varepsilon^{\prime\prime}, N}^{(p_2, q_2)}|
            \cdot
            \big|
      \Wg_N ( p, q ) &
         - 
           \Wg_N( p_1, q_1 ) 
           \cdot
             \Wg_N( p_2, q_2)       
            \big|\\
            \leq &
   \big|
  D_{ p, q } -
          C_{ p_1, q_1} D_{ p_2. q_2} - D_{p_1, q_1} C_{p_2, q_2}
          \big|
          + o\big(N^{0}\big).        
  \end{align*}   
  
  Next, fix
  $ p, q \in \cP_2^{ \varepsilon }(m, r) $.
  It suffices to show that
  \begin{align}\label{b:ineq}
  | \mathcal{B}_{\overrightarrow{\sigma_N}, \varepsilon, N }^{(m, r)} (p, q) |
   \leq
    N^{ m + r -|p \vee q |}
  \end{align}
  and we will obtain that
  \begin{align*}
  \big|
  \Wg_N (p, q) 
  \big|
  \cdot
  | \mathcal{B}_{\overrightarrow{\sigma_N}, \varepsilon, N }^{(m, r)} ( p, q )
  |
  \leq
  | C_{p, q}| + O(N^{-2}).
  \end{align*}
  
  Since 
  $ | \mathcal{B}_{\overrightarrow{\sigma_N}, \varepsilon, N }^{(m, r)} (p, q)  | $
  is invariant under circular permutations of the set 
  $ [ m ] $ 
  and of the set
  $ [ m + r ] \setminus [ r ] $,
  we can suppose that 
  $ m+ 1\in \{ p(m), q(m) \} $.
  In particular, 
  $ m $ is not the largest element in a cycle of
  $p\vee q $.
  
 For 
 $ s \in [ m + r ] $,
 denote
  \begin{align*}
 f(s) = | \{(i_1, j_1, \dots, i_s, j_s):\ \textrm{there exists }
 &
 (i_{s+1}, j_{s+ 1}, \dots, i_{ m + r}, j_{m + r}) \\
 &\textrm{ such that } 
 (i_1, j_1, \dots, i_{ m + r}, j_{m + r }) \in \mathcal{B}_{\overrightarrow{\sigma_N}, \varepsilon, N }^{(m, r)}(p, q) \} |. 
   \end{align*}
   
  The argument from the proof of Lemma \ref{fpq:1} gives that
 \begin{align*}
    f(m -1) \leq
     N^{m -| \{ B \in p \vee q \ :\ \max(B) \leq m -1\} | }.
 \end{align*} 
 Also, by construction,
 $( i_m, j_ m) = (j_{ m -1}, i_1 ) $,
 therefore
  $
   f(m) \leq f( m -1) 
  $.
  
   Suppose that 
   $ m+ 1 $ 
   is the largest element in a cycle 
   $ B $
   of 
   $ p \vee q $.
   Then the argument from the proof of Lemma \ref{fpq:1}(2) gives that
   \begin{align*}
   f(m+1) = f(m) \leq N^{m + 1-| \{ B \in p \vee q : \ \max(B) \leq m +1\} | }
   \end{align*}
   and
   (\ref{b:ineq}) follows as in Lemma \ref{fpq:1}(3).
   
   If 
   $ m + 1 $ 
   is not the largest element of any cycle of 
   $ p \vee q $, utilizing
   the assumption 
   $ m+ 1\in \{ p(m), q(m) \} $
   gives that either
   $ k_{m+1} = k_m  $
   or 
   $ l_{m +1}  = l_m $,
   so
   $ f(m + 1) \leq N \cdot f(m),$
   that is 
    \begin{align*}
     f(m+1) \leq N^{m + 1-| \{ B \in p \vee q  \ : \  \max(B) \leq m +1\} | }
     \end{align*}
     and again
     (\ref{b:ineq}) follows as in Lemma \ref{fpq:1}(3). 
   
 The conclusion of the Lemma follows taking
 \begin{align*}
 C(m, r, \varepsilon) = 
  \sum_{ p, q \in \cP_2^{\varepsilon}( m , r)}
   \big| C_{ p, q} |
   +
  \sum_{
      \substack{
      p_1, q_1 \in \cP_2^{ \varepsilon^{ \prime }} ( m )  \\
      p_2, q_2 \in
        \cP_2^{ \varepsilon^{ \prime \prime}}( r )}
        } 
    \big|
     D_{ p_1 \oplus p_2, q_1 \oplus q_2 } -
             C_{ p_1, q_1} D_{ p_2. q_2} - D_{p_1, q_1} C_{p_2, q_2}
             \big|.
 \end{align*}
  \end{proof}

\begin{remark}
Similar argument as the one discussed at the beginning of this section can be used to show that Lemma \ref{lemma:cov} implies that the conclusion of Corollary \ref{cor:36} holds almost surely. 

\end{remark}

\section{Random permutations satisfy (C1), (C2) and (C3)}

In this section we show that a sequence of random permutations, chosen uniformly satisfy conditions (C),
 which proves Theorem \ref{thm:13} of the Introduction. We also show that condition
 (C3)
  is satisfied almost surely by two independent sequences of uniformly chosen permutations. By a uniformly chosen random permutation in $\mathcal{S}\left([N]^2\right)$ we understand a random permutation $\mathfrak{S}_N$ such that for any $\sigma_N\in\mathcal{S}\left([N]^2\right)$ one has
\begin{align*}
\P\left(\mathfrak{S}_N=\sigma_N\right)=\frac{1}{N^2!}.
\end{align*}

\begin{proposition}\label{prop:11}
	Consider a sequence of random permutations $(\mathfrak{S}_N)_{N\geq 1}$, such that for each $N\geq 1$ the permutation $\mathfrak{S}_N$  acts on the set $[N]\times[N]$. Then we have
	\begin{align}\label{eq:Prop22.2}
	\frac{Y_N(\mathfrak{S}_N,\mathfrak{S}_N) }{N^3}\to 0 \mbox{ almost surely as } N\to \infty.
	\end{align}
\end{proposition}
\begin{remark}\label{proofofthm1.3}
The statement of the above proposition gives immediately that $(\mathfrak{S}_N)_{N\geq1}$ satisfies (C1) and (C2) almost surely, thus proving Theorem \ref{thm:13} of the Introduction.
\end{remark}

\begin{proof}[Proof of Proposition \ref{prop:11}]
	To makes formulas more transparent we will write $X_N$ and $Z_N$ for $X_N(\mathfrak{S}_N,id_N)$ and $Z_N(\mathfrak{S}_N)$ respectively. 

	Recall the definition of the random variable $Y_N$  
\begin{enumerate}
	\item[${}$]$ Y_{1,N}(\sigma_N, \sigma_N) =
	| \big\{  
	(i, j, k) \in [ N ]^3:\
	\pi_1 \circ \sigma_N (i, j) \in 
	\{ 
	\pi_1 \circ \sigma_N (i, k), 
	\pi_1 \circ \sigma_N ( k, j)   \} 
	\big\} |, $
	\item[${}$]
	$ Y_{2,N}(\sigma_N, \sigma_N) =
	| \big\{  
	(i, j, k) \in [ N ]^3:\
	\pi_2 \circ \sigma_N (i, j) \in 
	\{ 
	\pi_2 \circ \sigma_N (i, k), 
	\pi_2 \circ \sigma_N ( k, j)   \} 
	\big\} |, $
\end{enumerate}
and
\[ 
Y( \sigma_N, \sigma_N ) = Y_{1,N}(\sigma_N, \sigma_N ) 
+ Y_{2,N} ( \sigma_N , \sigma_N ).
\]	Observe that there is always an obvious choice which counts in the sum namely $k=i$ and $l=j$, but there are only $N^2$ such choices, let us first define new random variables:
	\begin{itemize}
		\item $Y^{(rr)}_N(\sigma_N)$ counts number of elements which were in the same row before the permutation and are in the same row after the permutation,
		\item $Y^{(rc)}_N(\sigma_N)$ counts number of elements which were in the same row before the permutation and are in the same column after the permutation,
		\item $Y^{(cr)}_N(\sigma_N)$ counts number of elements which were in the same column before the permutation and are in the same row after the permutation,
		\item $Y^{(cc)}_N(\sigma_N)$ counts number of elements which were in the same column before the permutation and are in the same column after the permutation.
	\end{itemize}
	 Then
	 \begin{align*}
	 	Y_N(\sigma_N,\sigma_N)=N^2+Y^{(rr)}_N(\sigma_N)+Y^{(rc)}_N(\sigma_N)+Y^{cr}_N(\sigma_N)+Y^{cc}_N(\sigma_N).
	 \end{align*}

	 We shall only show that $Y^{(rr)}_N(\mathfrak{S}_N)/N^2\to 1/2$  almost surely. A similar proof shows that remaining three variables also converge a.s. to $1/2$. In what follows we suppress in notation $\mathfrak{S}_N$, and write e.g. $Y^{(rr)}_N$ to mean the random variable $Y^{(rr)}_N(\mathfrak{S}_N)$. Our strategy is to use Chebyshev's inequality, that is we write
	\begin{align*}
	\P\left(\left|\frac{Y^{(rr)}_N-\E Y^{(rr)}_N}{N^2}\right|>\varepsilon \right)\leq\frac{\mathrm{Var}(Y^{(rr)}_N)}{\varepsilon^2N^4}
	\end{align*}
	and it suffices to show that $\mathrm{Var}(Y^{(rr)}_N(\mathfrak{S}_N))=O(N^2)$ and $\E Y_N^{(rr)}/N^2\to 1/2$.
	
	For $j,l,k\in[N]$ such that $k<l$ denote by $I^{(j)}_{k,l}$ the indicator of the event that elements $(k,j)$ and $(l,j)$ after $\mathfrak{S}_N$ are in the same row. Then one has
	\begin{align*}
	Y_N^{(rr)}=\sum_{j=1}^N\sum_{\substack{1\leq k<l\leq N}} I^{(j)}_{k,l}.
	\end{align*}
	We shall use this representation in order to calculate the expectation and estimate the variance of $Y_N^{(rr)}$. 
	
	Observe that for two elements to go to the same row we have: $N$ choices for the row, then $2 \binom{N}{2}$ choices for placing the elements in the chosen row. On the other hand we have $2 \binom{N^2}{2}$ of all possibilities of placing two fixed elements in the matrix, thus we have
	\begin{align*}
	\E\left(I^{(j)}_{k,l}\right)=\frac{N\binom{N}{2} 2}{\binom{N^2}{2} 2 }=\frac{1}{N+1}.
	\end{align*}
	Hence
	\begin{align*}
	\E Y_N^{(rr)}=\sum_{j=1}^N\sum_{\substack{1\leq k<l\leq N}} \E I^{(j)}_{k,l}=\sum_{j=1}^N\sum_{\substack{1\leq k<l\leq N}} \frac{1}{N+1}=\frac{N}{N+1} \binom{N}{2}.
	\end{align*}
	thus we proved that $\E Y_N^{(r)}/N^2\to 1/2$ as $N \to \infty$.
	
	It remains to show that $\mathrm{Var}=O(N^2)$. Of course we have	\begin{align}\label{eqn:var}
	\mathrm{Var}(Y_N^{(rr)})= \sum_{j=1}^N\sum_{\substack{1\leq k<l\leq N}} \mathrm{Var}\left(I^{(j)}_{k,l}\right)+\sum_{(j_1,k_1,l_1)\neq (j_1,k_1,l_1)} \mathrm{Cov}\left( I^{j_1}_{l_1,k_1}, I^{j_2}_{l_2,k_2}\right).
	\end{align}
	Since 
	\begin{align*}
	\E\left(\left(I^{(j)}_{k,l}\right)^2\right)=\E\left(I^{(j)}_{k,l}\right)=\frac{1}{N+1}
	\end{align*}
	we have $\mathrm{Var}\left(I^{(j)}_{k,l}\right)\leq 1/(N+1)$. Thus
	\begin{align*}
	\sum_{j=1}^N\sum_{\substack{1\leq k<l\leq N}} \mathrm{Var}\left(I^{(j)}_{k,l}\right)=O(N^2).
	\end{align*}
	
	Let us calculate the covariances. Thus we have to calculate $\E(I^{j_1}_{k_1,l_1}I^{j_2}_{k_2,l_2})$, we consider two cases: 
	
	$1^\mathrm{o}$ The two pairs $\{(k_1,j_1),(l_1,j_1)\}$ and $\{(k_2,j_2),(l_2,j_2)\}$ are different, that is $j_1 \neq j_2$ or ($j_1=j_2$ and $k_1\neq k_2$ and $l_1\neq l_2$). Remember that $k_1<l_1$ and $k_2<l_2$, so this condition indeed is equivalent to the fact that the two pairs to be different.
	
	$2^\mathrm{o}$ There is one common element in the two pairs that is the pairs are of the form $\{(j,k_1),(j,l_1)\}$ and $\{(j,k_2),(j,l_2)\}$ and we have $k_2\in\{k_1,l_1\}$ or $l_2\in\{k_1,l_1\}$. The second case says that in fact we have three elements from the same column and we ask for probability that all three will land after the permutation in the same row.
	
	Consider the case $1^\mathrm{o}$ then, we have two different pairs thus we have four elements of a matrix, that after the permutation are supposed to be placed such that elements of each pair are in the same row, it could be that all four elements are in the same row and we have $N\binom{N}{4} 4!$ possibilities for this or each pair is in different row and we have $\binom{N}{2}^3 (2!)^3$ possibilities. We have $\binom{N^2}{4} 4!$ choices to place $4$ elements in an $N\times N$ matrix thus we get 
	\begin{align*}
	\E I^{(j_1)}_{k_1,l_1} I^{(j_2)}_{k_2,l_2}=\frac{N\binom{N}{4} 4!+ \binom{N}{2}^3 (2!)^3}{\binom{N^2}{4}4!}=\frac{N^3-N^2-4N+6}{(N+1)(N^2-2)(N^2-3)}.
	\end{align*}
	Calculating the covariance we get
	
	\begin{eqnarray*}
	\mathrm{Cov}\left(I^{(j_1)}_{k_1,l_1},I^{(j_2)}_{k_2,l_2}\right) &=&
	\frac{N^3-N^2-4N+6}{(N+1)(N^2-2)(N^2-3)}-\frac{1}{(N+1)^2} \\
	 & = & \frac{2 N}{(N+1)^2 \left(N^2-2)(N^2-3)\right)}.
	\end{eqnarray*}
	Thus in the case $1^\mathrm{o}$ we get that $\mathrm{Cov}\left(I^{(j_1)}_{k_1,l_1},I^{(j_2)}_{k_2,l_2}\right)=O(1/N^5)$ for $N>1$.

	Consider the case $2^\mathrm{o}$ then as explained above we only have three elements which were in the same column and after the permutation all three are supposed to be in the same row thus similarly as above we have
	\begin{align*}
	\E I^{(j_1)}_{k_1,l_1} I^{(j_2)}_{k_2,l_2}=\frac{N\binom{N}{3} 3!}{\binom{N^2}{3}3!}=\frac{N-2}{(N+1)(N^2-2)}.
	\end{align*}
	Calculating covariance we get
	\begin{align*}
	\mathrm{Cov}\left(I^{(j_1)}_{k_1,l_1},I^{(j_2)}_{k_2,l_2}\right)=
	\frac{N-2}{(N+1)(N^2-2)}-\frac{1}{(N+1)^2}=\frac{-N}{(N+1)^2 (N^2-2)}
	\end{align*}
	which is negative as soon as $N>1$, so we can omit this terms as we look for an upper bound for variance.
	
	Now observe that we have the sum of covariances in \eqref{eqn:var} contains $\binom{N^2 (N-1)/2}{2}$ terms, as we have $N^2 (N-1)/2$ pairs which were in the same row and now we choose pairs among them. Hence we have
	\begin{align*}
	\sum_{(j_1,k_1,l_1)\neq (j_1,k_1,l_1)} \mathrm{Cov}\left( I^{j_1}_{l_1,k_1}, I^{j_2}_{l_2,k_2}\right)=O(N),
	\end{align*}
	which completes the proof of the fact that $Y_N^{rr}\to 1/2$ a.s. with $N\to\infty$.
	
\end{proof}

\begin{proposition}\label{prop:73}
	Consider two sequence of random permutations $(\mathfrak{S}^{(1)}_N)_{N\geq 1},(\mathfrak{S}^{(2)}_N)_{N\geq 1}$, such that for each $N\geq 1$ the permutation $\mathfrak{S^{(i)}}_N$ acts on the set $[N]\times[N]$, for $i=1,2$. Then we have
	\begin{align}\label{eqn:29}
	\frac{X_N(\mathfrak{S}^{(1)}_N,\mathfrak{S}^{(2)}_N) }{N^2}\to 0 \mbox{ almost surely as } N\to \infty.
	\end{align}
	\begin{align}\label{eqn:30}
	\frac{Y_N(\mathfrak{S}^{(1)}_N,\mathfrak{S}^{(2)}_N)}{N^3}\to 1 \mbox{ almost surely as } N\to \infty.
	\end{align}
\end{proposition}
\begin{corollary}
	The statement of the above proposition gives immediately that independent sequences of permutation satisfy condition 
	(C3).
	 Thus one can formulate Theorem \ref{thm:14} in almost sure sense with independent sequences of random permutations similarly as Theorem \ref{thm:13}.
\end{corollary}
\begin{proof}[Proof of Proposition \ref{prop:73}]
	In order to prove \eqref{eqn:29} it is enough to observe that $\left(\mathfrak{S}^{(2)}_N\right)^{-1}\circ\mathfrak{S}^{(1)}_N$ is also a uniformly distributed permutation and recall that Proposition \ref{prop:11} implies that 
	(C2)
	 is almost surely satisfied by a sequence of random permutations.
	
	For the proof of \eqref{eqn:29} one proceeds similarly as in the proof of Proposition \ref{prop:11}.
\end{proof}


\end{document}